\documentclass[11pt]{article}
\usepackage{amsmath,amsfonts,amssymb,amsthm,enumerate,graphicx}
\usepackage{tikz,authblk}
\usepackage{subfig}
\usepackage{url}

\newtheorem{theorem}{{\bf Theorem}}[section]

\newtheorem{corollary}[theorem]{{\bf Corollary}}

\newtheorem{example}[theorem]{{\bf Example}}
\newtheorem{lemma}[theorem]{{\bf Lemma}}

\newtheorem{proposition}[theorem]{{\bf Proposition}}

\newcommand{\ZZ}{ \ensuremath{\mathbb{Z}}}

\begin{document}

\author[1] {Basudeb Datta}
\author[2] { Dipendu Maity}

\affil[1,2]{Department of Mathematics, Indian Institute of Science, Bangalore 560\,012, India.
 ${}^1$dattab@iisc.ac.in, ${}^2$dipendumaity@iisc.ac.in.}

\title{Semi-equivelar and vertex-transitive maps on the torus}


\date{`{\em Beitr Algebra Geom} (2017) {\bf 58}:617--634'}

\maketitle

\vspace{-10mm}

\begin{abstract}
A vertex-transitive map $X$ is a map on a closed surface on which the automorphism group ${\rm Aut}(X)$ acts transitively on the set of vertices.
If the face-cycles at all the vertices in a map are of same type then the map is said to be a semi-equivelar map. Clearly, a vertex-transitive map is semi-equivelar. Converse of this is not true in general. We show that there are eleven types of semi-equivelar maps on the torus. Three of these are equivelar maps. It is known that two of these three types are always vertex-transitive. We show that this is true for the remaining one type of equivelar maps and one other type of semi-equivelar maps, namely, if $X$ is a semi-equivelar map of type $[6^3]$ or $[3^3, 4^2]$ then $X$ is vertex-transitive. We also show, by presenting examples, that this result is not true for the remaining seven types of semi-equivelar maps. There are ten types of semi-equivelar maps on the Klein bottle. We present examples in each of the ten types which are not vertex-transitive.
\end{abstract}

\noindent {\small {\em MSC 2010\,:} 52C20, 52B70, 51M20, 57M60.

}

\section{Introduction}

By a map we mean a polyhedral map on a surface. So, a face of a map is a $p$-gon for some integer $p\geq 3$. A map $X$ is said to be {\em weakly regular} or {\em vertex-transitive} if the automorphism group ${\rm Aut}(X)$ acts transitively on the set $V(X)$ of vertices of $X$.

If $v$ is a vertex in a map $X$ then the faces containing $v$ form a cycle (called the {\em face-cycle}) $C_v$ in the dual graph $\Lambda(X)$ of $X$. Clearly, $C_v$ is of the form $P_1\mbox{-}P_2\mbox{-}\cdots\mbox{-}P_k\mbox{-}P_1$, where $P_i$ is a path consisting of $n_i$ $p_i$-gons and $p_i\neq p_{i+1}$ for $1\leq i\leq k$ (addition in the suffix is modulo $k$). A map $X$ is called {\em semi-equivelar} (or {\em semi-regular}) if $C_u$ and $C_v$ are of same type for any two vertices $u$ and $v$ of $X$. More precisely, there exist natural numbers $p_1, \dots, p_k\geq 3$ and $n_1, \dots, n_k\geq 1$, $p_i\neq p_{i+1}$ such that both $C_u$ and $C_v$ are of the form $P_1\mbox{-}P_2\mbox{-}\cdots\mbox{-}P_k\mbox{-}P_1$ as above, where $P_i$ is a path consisting of $n_i$ $p_i$-gons. In this case, we say that $X$ is {\em semi-equivelar of type} $[p_1^{n_1}, \dots, p_k^{n_k}]$. (We identify $[p_1^{n_1}, \dots, p_k^{n_k}]$ with $[p_2^{n_2}, \dots, p_k^{n_k}, p_1^{n_1}]$ and with $[p_k^{n_k}, \dots, p_1^{n_1}]$.)
An {\em equivelar} map (of type $[p^{q}]$, $(p, q)$ or $\{p, q\}$) is a semi-equivelar map of type $[p^{q}]$ for some $p, q\geq 3$.
Clearly, a vertex-transitive map is semi-equivelar.


A {\em semi-regular} tiling of the plane $\mathbb{R}^2$ is a tiling of $\mathbb{R}^2$ by regular polygons such that all the vertices of the tiling are of same type.
A {\em semi-regular} tiling of $\mathbb{R}^2$ is also known as {\em Archimedean}, or {\em homogeneous}, or {\em uniform} tiling. In \cite{GS1977}, Gr\"{u}nbaum and Shephard showed that there are exactly eleven types of Archimedean tilings on the plane. These types are $[3^6]$, $[3^4,6^1]$, $[3^3,4^2]$,  $[3^2,4^1,3^1,4^1]$, $[3^1,6^1,3^1,6^1]$, $[3^1,4^1,6^1,4^1]$, $[3^1,12^2]$, $[4^4]$, $[4^1,6^1,12^1]$, $[4^1,8^2]$, $[6^3]$.
Clearly, a {\em semi-regular} tiling on $\mathbb{R}^2$ gives a semi-equivelar map on $\mathbb{R}^2$. But, there are semi-equivelar maps on the plane which are not (not isomorphic to) an Archimedean tiling. In fact, there exists $[p^q]$ equivelar maps on $\mathbb{R}^2$ whenever $1/p+1/q<1/2$ (e.g., \cite{CM1957}, \cite{FT1965}). Thus, we have

\begin{proposition} \label{prop:plane}
There are infinitely many types of equivelar maps on the plane $\mathbb{R}^2$.
\end{proposition}

All vertex-transitive maps on the 2-sphere are known. These are the boundaries of Platonic and Archimedean solids and two infinite families of types (namely, of types $[4^2, n^1]$ and $[3^3,m^1]$ for $4\neq n\geq 3$, $m\geq 4$) \cite{GS1977}. Similarly, there are infinitely many types of vertex-transitive maps on the real projective plane \cite{Ba1991}. Thus, there are infinitely many types of semi-equivelar maps on the 2-sphere and the real projective plane.
But, for a surface of negative Euler characteristic the picture is different. In \cite{Ba1991}, Babai has shown the following.

\begin{proposition} \label{prop:Babai}
A semi-equivelar map on a surface of Euler characteristic $\chi<0$ has at most $-84\chi$ vertices.
\end{proposition}

As a consequence of this we get

\begin{corollary} \label{coro:finiteness}
If the Euler characteristic $\chi(M)$ of a surface $M$ is negative then the number of semi-equivelar maps on $M$ is finite.
\end{corollary}

We know from \cite{DN2001} and \cite{DU2005} that infinitely many equivelar maps exist on both the torus and the Klein bottle. Thus, infinitely many semi-equivelar maps exist on both the torus and the Klein bottle. But, only eleven types of semi-equivelar maps on the torus and ten types of semi-equivelar maps on the Klein bottle are known in the literature. All these are quotients of Archimedean tilings of the plane \cite{Ba1991, Su2011t, Su2011kb}.
Since there are infinitely many equivelar maps on the plane, it is natural to ask whether there are more types of semi-equivelar maps on the torus or the Klein bottle.  Here we prove

\begin{theorem} \label{theo:GrSh}
Let $X$ be a semi-equivelar map on a surface $M$. {\rm (a)} If $M$ is the torus then the type of $X$ is $[3^6]$, $[6^3]$, $[4^4]$, $[3^4,6^1]$, $[3^3,4^2]$, $[3^2,4^1,3^1,4^1]$,  $[3^1,6^1,3^1,6^1]$,
$[3^1,4^1,6^1,4^1]$,  $[3^1,12^2]$, $[4^1,8^2]$  or $[4^1,6^1,12^1]$.
{\rm (b)} If $M$ is the Klein bottle then the type of $X$ is $[3^6]$, $[6^3]$, $[4^4]$, $[3^3,4^2]$, $[3^2,4^1,3^1,4^1]$,  $[3^1,6^1,3^1,6^1]$,
$[3^1,4^1,6^1,4^1]$,  $[3^1,12^2]$, $[4^1,8^2]$  or $[4^1,6^1,12^1]$.
\end{theorem}

Theorem \ref{theo:GrSh} and the known examples (also the examples in Section 4) imply that there are exactly eleven types of semi-equivelar maps on the torus and ten types of semi-equivelar maps on the Klein bottle.

In \cite{BK2008}, Brehm and K\"{u}hnel presented a formula to determine the number of distinct vertex-transitive equivelar maps of types $[3^6]$ and $[4^4]$ on the torus.
It was shown in \cite{DU2005} that every equivelar map of type $[3^6]$ on the torus is vertex-transitive. By the similar arguments, one can easily show that a equivelar map of type $[4^4]$ on the torus is vertex-transitive (also see \cite[Proposition 6]{BK2008}). Thus, we have

\begin{proposition} \label{prop:36&44}
Let $X$ be an equivelar map on the torus. If the type of $X$ is $[3^6]$ or $[4^4]$ then $X$ is vertex-transitive.
\end{proposition}

Here we prove

\begin{theorem} \label{theo:torus}
Let $X$ be a semi-equivelar map on the torus. If the type of $X$ is $[6^3]$ or    $[3^3,4^2]$ then $X$ is vertex-transitive.
\end{theorem}

In Section \ref{sec:examples}, we present examples of the other seven types of semi-equivelar maps which are not vertex-transitive. This proves

\begin{theorem} \label{theo:non-wr-torus}
If $[p_1^{n_1}, \dots, p_k^{n_k}] =$ $[3^2,4^1,3^1,4^1]$, $[3^4,6^1]$, $[3^1,6^1,3^1,6^1]$, $[3^1,4^1,6^1,4^1]$, \linebreak  $[3^1,12^2]$, $[4^1,8^2]$  or $[4^1,6^1,12^1]$ then there exists a semi-equivelar map of type $[p_1^{n_1}, \dots, p_k^{n_k}]$ on the torus which is not vertex-transitive.
\end{theorem}

 In \cite{DU2005}, the first author and Upadhyay have presented examples of $[3^6]$ equivelar maps on the Klein bottle which are not vertex-transitive. In Section \ref{sec:examples}, we present examples of the other nine types of semi-equivelar maps on the Klein bottle which are not vertex-transitive. Thus, we have

\begin{theorem} \label{theo:non-wr-kbottle}
If $[p_1^{n_1}, \dots, p_k^{n_k}]$ is one in the list of $10$ types in Theorem $\ref{theo:GrSh}$~{\rm (b)} then there exists a semi-equivelar map of type $[p_1^{n_1}, \dots, p_k^{n_k}]$ on the Klein bottle which is not vertex-transitive.
\end{theorem}

If the type of a semi-equivelar map $X$ on the torus is different from $[3^3,4^2]$ then, by Theorem \ref{theo:non-wr-torus}, the vertices of $X$ may form more than one ${\rm Aut}(X)$-orbits. Here we prove

\begin{theorem} \label{theo:torus33434&3636}
Let $X$ be a semi-equivelar map on the torus. Let the vertices of $X$ form $m$ ${\rm Aut}(X)$-orbits. {\rm (a)} If the type of $X$ is  $[3^2,4^1,3^1,4^1]$ then $m\leq 2$.
{\rm (b)} If the type of $X$ is  $[3^1,6^1,3^1,6^1]$ then $m\leq 3$.
\end{theorem}


Several examples of $[3^6]$ and $[4^4]$ equivelar maps on the torus are in \cite{DN2001}. From this, one can construct equivelar maps of type $[6^3]$ on the torus. In Example \ref{eg:8maps-torus}, we also present an equivelar map of type $[3^3,4^2]$ on the torus for the sake of completeness.

\section{Proofs of Theorem \ref{theo:GrSh} and Proposition  \ref{prop:plane}} \label{sec:proofs-1}

For $n\geq 3$, the $n$-gon whose edges are $u_1u_2, \dots, u_{n-1}u_{n}, u_{n}u_{1}$ is denoted by $u_1\mbox{-}u_2\mbox{-}\cdots\mbox{-}u_n\mbox{-}u_1$ or by $C_n(u_1, \dots, u_n)$. We call 3-gons and 4-gons by {\em triangles} and {\em quadrangles} respectively. A triangle $u\mbox{-}v\mbox{-}w\mbox{-}u$ is also denoted by $uvw$.
If $X$ is a map on a surface $M$ then we identify a face of $X$ in $M$ with the boundary cycle of the face.

\begin{proof}[Proof of Proposition \ref{prop:plane}]
In \cite{DN2001}, it was shown that there exists equivelar map of type $[3^8]$ on the orientable surface of genus $g$ for each $g\geq 4$. For a fixed $g\geq 4$, let $X$ be
one such equivelar map of type $[3^8]$ on the surface $M_g$ of genus $g$. Since the 2-disk $\mathbb{D}^2$ is the universal cover of $M_g$, by pulling back $X$, we get an
equivelar map $\widetilde{X}$ of type $[3^8]$ on $\mathbb{D}^2$ and a polyhedral map $\eta : \widetilde{X} \to X$. From the constructions in \cite{Da2005} and \cite{DN2001}, we know that an equivelar map  of type $[p^q]$ exists on some surface (orientable or non-orientable) of appropriate genus for each $[p^q]$ in $\{[3^7], [4^5], [4^6], [3^{3\ell-1}], [3^{3\ell}], [k^k] \, : \ell\geq 3, k\geq 5\}$. So, by the same arguments, equivelar maps  of types $[p^q]$ exist on $\mathbb{D}^2$ for $[p^q]$ in $\{[3^7], [4^5], [4^6], [3^{3\ell-1}], [3^{3\ell}], [k^k] \, : \ell\geq 3, k\geq 5\}$. More generally, there exist equivelar maps of type $[p^q]$ on $\mathbb{D}^2$ whenever $1/p+1/q<1/2$ (cf., \cite{CM1957}, \cite{FT1965}, \cite{GS1977}).
Since $\mathbb{R}^2$ is homeomorphic to $\mathbb{D}^2$, an equivelar map of type $[p^q]$ determines an equivelar map of type $[p^q]$ on $\mathbb{R}^2$.
Thus, there exist equivelar maps of types $[p^q]$ on $\mathbb{R}^2$ whenever $1/p+1/q<1/2$. The result now follows.
\end{proof}

\begin{lemma}\label{lemma:21types}
Let $X$ be a semi-equivelar map on a surface $M$. If $\chi(M) = 0$ then the type of $X$ is $[3^6]$, $[3^4,6^1]$, $[3^3,4^2]$, $[3^2,4^1,3^1,4^1]$, $[4^4]$, $[3^1,6^1,3^1, 6^1]$, $[3^2,6^2]$, $[3^2,4^1,12^1]$, $[3^1,4^1,3^1,12^1]$, $[3^1,4^1,6^1,4^1]$, $[3^1,4^2,6^1]$, $[6^3]$, $[3^1,12^2]$, $[4^1,8^2]$, $[5^2,10^1]$,  $[3^1,7^1,42^1]$, \linebreak $[3^1,8^1,24^1]$, $[3^1,9^1,18^1]$, $[3^1,10^1,15^1]$, $[4^1,5^1,20^1]$ or $[4^1,6^1,12^1]$.
\end{lemma}

\begin{proof}
Let the type of $X$ be $[p_1^{n_1}, \dots, p_k^{n_k}]$.  Consider the $\ell$-tuple  $(q_1^{m_1}, \dots, q_{\ell}^{m_{\ell}})$, where $q_i\neq q_j$ for $i\neq j$, $q_i = p_j$ for some $j$, $m_i = \sum_{p_i=q_j}n_j$ and $(m_1,q_1) > (m_2, q_2) > \cdots >  (m_{\ell}, q_{\ell})$. (Here, $(m, p) > (n, q)$ means either (i) $m>n$ or (ii) $m=n$ and $p < q$.)

\medskip

\noindent {\em Claim.} $(q_1^{m_1}, \dots, q_{\ell}^{m_{\ell}}) = (3^6)$, $(3^4,6^1)$, $(3^3,4^2)$, $(4^4)$, $(3^2,6^2)$, $(3^2,4^1,12^1)$, $(4^2, 3^1,6^1)$, $(6^3)$, $(12^2,3^1)$, $(8^2,4^1)$, $(5^2,10^1)$,  $(3^1,7^1,42^1)$,  $(3^1,8^1,24^1)$, $(3^1,9^1,18^1)$, $(3^1,10^1,15^1)$, $(4^1,5^1,20^1)$ or $(4^1,6^1,12^1)$.

\smallskip

Let $f_0, f_1$ and $f_2$ denote the number of vertices, edges and faces of $X$ respectively. Let $d$ be the degree of each vertex. Then, $d = n_1 + \cdots + n_k =  m_1 + \cdots + m_{\ell}$ and $f_1 = f_0\times d/2$. Clearly, the number of $q_i$-gons is $f_0\times m_i/q_i$. This implies that $f_2 = f_0(m_1/q_1 + \cdots + m_{\ell}/q_{\ell})$. Since $\chi(M)=0$, it follows that $f_0 - f_0(m_1 + \cdots + m_{\ell})/2 + f_0(m_1/q_1 + \cdots + m_{\ell}/q_{\ell}) = 0$ or
\begin{align}\label{eq:first}
  \left(\frac{1}{2} - \frac{1}{q_1}\right)m_1 + \cdots + \left(\frac{1}{2} -
  \frac{1}{q_{\ell}}\right)m_{\ell} = 1.
\end{align}

Since $q_i\geq 3$, it follows that $d\leq 6$. Moreover, if $d=6$ then $\ell=1$ and $q_1=3$. In this case, $(q_1^{m_1}, \dots, q_{\ell}^{m_{\ell}}) = (3^6)$.

Now, assume $d=5$. Then $(m_1, \dots, m_{\ell})= (5), (4,1), (3,2), (3,1,1), (2, 2, 1), (2,1,1, 1)$ or $(1,1,1,1,1)$. It is easy to see that for $(m_1, \dots, m_{\ell})= (5), (3,1,1), (2, 2, 1), (2,1,1, 1)$ or $(1,1,1,1,1)$, Eq.  \eqref{eq:first} has no solution. So, $(m_1, \dots, m_{\ell}) = (4,1)$ or $(3,2)$. In the first case, $(q_1, q_2) = (3,6)$ and in the second case, $(q_1, q_2) = (3,4)$.  Thus, $(q_1^{m_1}, \dots, q_{\ell}^{m_{\ell}}) = (3^4, 6^1)$ or $(3^3, 4^2)$.

Let $d=4$. Then $(m_1, \dots, m_{\ell})= (4), (3,1), (2,2), (2,1,1)$ or $(1,1,1,1)$.
Again, for $(m_1, \dots, m_{\ell})= (3,1)$ or $(1,1,1,1)$, Eq.  \eqref{eq:first} has no solution.
So, $(m_1, \dots, m_{\ell}) = (4)$, $(2,2)$ or $(2,1,1)$. In the first case, $q_1 = 4$,  in the second case, $(q_1, q_2) = (3,6)$ and in the third case, $(q_1, \{q_2, q_3\}) = (3, \{4, 12\})$ or $(4, \{3,6\})$.  Thus, $(q_1^{m_1}, \dots, q_{\ell}^{m_{\ell}}) = (4^4), (3^2, 6^2), (3^2,4^1,12^1)$ or $(4^2, 3^1, 6^1)$.

Finally, assume $d=3$. Then $(m_1, \dots, m_{\ell})= (3), (2,1)$ or $(1,1,1)$.
In the first case, $q_1 = 6$,  in the second case, $(q_1, q_2) = (12,3), (8,4)$ or $(5,10)$ and in the third case, $\{q_1, q_2, q_3\} = \{3, 7, 42\}$, $\{3,8,24\}$, $\{3,9,18\}$, $\{3,10, 15\}$, $\{4,5,20\}$ or $\{4, 6,12\}$.  Thus, $(q_1^{m_1}, \dots, q_{\ell}^{m_{\ell}}) = (6^3)$, $(12^2, 3^1)$, $(8^2,4^1)$, $(5^2, 10^1)$, $(3^1,7^1,42^1)$, $(3^1,8^1,24^1)$, $(3^1,9^1,18^1)$, $(3^1,10^1,15^1)$, $(4^1, 5^1$,  $20^1)$ or $(4^1,6^1,12^1)$. This proves the claim.

The lemma follows from the claim.
\end{proof}

We need the following technical lemma for the proof of Theorem \ref{theo:GrSh}.


\begin{lemma}\label{lemma:3restrictions}
If $[p_1^{n_1}, \dots, p_k^{n_k}]$ satisfies any of the following three properties then $[p_1^{n_1}$, $\dots, p_k^{n_k}]$ can not be the type of any semi-equivelar map on a surface.
\begin{enumerate}[{\rm (i)}]
\item There exists $i$ such that $n_i=2$, $p_i$ is odd and $p_j\neq p_i$ for all $j\neq i$.

\item There exists $i$ such that $n_i=1$, $p_i$ is odd, $p_j\neq p_i$ for all $j\neq i$ and $p_{i-1}\neq p_{i+1}$. (Here, addition in the subscripts are modulo $k$.) 

\item $[p_1^{n_1}, \dots, p_k^{n_k}]$ is of the form $[p^1, q^m, p^1, r^n]$, where $p, q, r$ are distinct and $p$ is odd.

\end{enumerate}

\end{lemma}

\begin{proof}
If possible let there exist a semi-equivelar map $X$ of type $[p_1^{n_1}, \dots, p_k^{n_k}]$ which satisfies (i). Let $A = u_1\mbox{-}u_2\mbox{-}u_3\mbox{-}\cdots\mbox{-}u_{p_i}\mbox{-}u_1$ be a $p_i$-gon.
Let the other face containing $u_{r}u_{r+1}$ be $A_r$ for $1\leq r\leq p_i$. (Addition in the subscripts are modulo $p_i$.) Consider the face-cycle of the vertex $u_1$. Since $p_j\neq p_i$ for all $j\neq i$ and $n_i=2$, it follows that exactly one of $A_1$ and $A_{p_i}$ is a $p_i$-gon. Assume, without loss, that $A_1$ is a $p_i$-gon. Since $u_2$ is in two $p_i$-gons, it follows that $A_2$ is not a $p_i$-gon. Therefore (by considering the vertex $u_3$, as in the case for the vertex $u_1$), $A_3$ is a $p_i$-gon. Continuing this way, we get $A_1, A_3, A_5, \dots $ are $p_i$-gons. Since $p_i$ is odd, it follows that $A_{p_i}$ is a $p_i$-gon. Then we get three $p_i$-gons, namely, $A$, $A_1$ and $A_{p_i}$, through $u_1$. This is a contradiction.

Now, suppose there exists a semi-equivelar map $Y$ of type $[p_1^{n_1}, \dots, p_k^{n_k}]$ which satisfies (ii). Let $B = u_1\mbox{-}u_2\mbox{-}u_3\mbox{-}\cdots \mbox{-} u_{p_i}\mbox{-}u_1$ be a $p_i$-gon. Let the other face containing $u_{r}u_{r+1}$ be $B_r$ for $1\leq r\leq p_i$. Consider the face-cycle of the vertex $u_2$. Since $p_j\neq p_i$ and $n_i=1$, $A$ is the only $p_i$-gon containing $u_2$.
Since $p_{i-1}\neq p_{i+1}$, it follows that one of $B_1$ and $B_{2}$ is a $p_{i-1}$-gon and the other is a $p_{i+1}$-gon. Assume, without loss, that $B_1$ is a $p_{i-1}$-gon and $B_2$ is a $p_{i+1}$-gon. Then, by the same argument as above, $B_1, B_3, B_5, \dots$ are $p_{i-1}$-gons and $B_2, B_4, \dots$ are $p_{i+1}$-gons. Since $p_i$ is odd, it follows that $B_{p_i}$ is a $p_{i-1}$-gon. Then, from the face-cycle of $u_1$, it follows that $p_{i+1}=p_{i-1}$. This contradicts the assumption.

Finally, assume that there exists a semi-equivelar map $Z$ of type $[p^1, q^m, p^1, r^n]$, where $p, q, r$ are distinct and $p$ odd. Let $P$ and $Q$ be two adjacent faces through a vertex $u_1$, where $P$ is a $p$-gon and $Q$ is a $r$-gon. Assume that
$P= u_1\mbox{-}u_2\mbox{-}u_3\mbox{-}\cdots \mbox{-} u_{p}\mbox{-}u_1$ and $Q= u_1\mbox{-}v_2\mbox{-}v_3\mbox{-}\cdots \mbox{-} v_{r-1}\mbox{-}u_{p} \mbox{-} u_1$.
Let the other face containing $u_{j}u_{j+1}$ be $P_j$ for $1\leq j\leq p$. (Addition in the subscripts are modulo $p$.) Since $p, q, r$ are distinct, considering the face-cycle of $u_1$, it follows that $P_1$ is a $q$-gon. Considering the face-cycle of $u_2$, by the similar argument (interchanging $r$ and $q$), it follows that $P_2$ is a $r$-gon. Continuing this way, we get $P_1, P_3, \dots$ are $q$-gons and $P_2, P_4, \dots$ are $r$-gons. Since $p$ is odd, it follows that $P_{p}$ is a $q$-gon. This is a contradiction since $P_{p} = Q$ is a $r$-gon and $r\neq q$. This completes the proof.
\end{proof}

In \cite[Theorem 2.1]{DU2005}, the second author and Upadhyay have proved the following.

\begin{proposition} \label{prop:3^46^1}
There is no semi-equivelar map of type $[3^4, 6^1]$ on the Klein bottle.
\end{proposition}

\begin{proof}[Proof of Theorem \ref{theo:GrSh}]
Let $X$ be a semi-equivelar map of type $[p_1^{n_1}, \dots, p_k^{n_k}]$ on the torus.
By Lemma \ref{lemma:3restrictions} (i), $[p_1^{n_1}, \dots, p_k^{n_k}] \neq [3^2,6^2], [3^2,4^1, 12^1], [5^2,10^1]$. By Lemma \ref{lemma:3restrictions} (ii), $[p_1^{n_1}, \dots,$ $p_k^{n_k}] \neq [3^1,4^2,6^1], [3^1,7^1,42^1], [3^1,8^1,24^1], [3^1,9^1,18^1]$, $[3^1,10^1,15^1]$, $[4^1,5^1,20^1]$. Also, by Lemma \ref{lemma:3restrictions} (iii), $[p_1^{n_1}, \dots, p_k^{n_k}]\neq
[3^1,4^1,3^1,12^1]$. The result now follows by Lemma \ref{lemma:21types}.

Let $X$ be a semi-equivelar map of type $[p_1^{n_1}, \dots, p_k^{n_k}]$ on the Klein bottle. As above, by Lemma \ref{lemma:3restrictions}, $[p_1^{n_1}, \dots, p_k^{n_k}]\neq [3^2,6^2], [3^2,4^1, 12^1], [5^2,10^1], [3^1,4^2,6^1]$, $[3^1,7^1,42^1]$, $[3^1,8^1,24^1]$, $[3^1,9^1,18^1]$, $[3^1,10^1,15^1]$, $[4^1,5^1,20^1]$, $[3^1,4^1,3^1,12^1]$. By Proposition \ref{prop:3^46^1}, $[p_1^{n_1}, \dots, p_k^{n_k}] \neq [3^4,6^1]$.
The result now follows by Lemma \ref{lemma:21types}.
\end{proof}

\section{Proof of Theorem \ref{theo:torus}}\label{sec:proofs-2}

A triangulation of a 2-manifold is called {\em degree-regular} if each of its vertices have the same degree. In other word, a degree-regular triangulation is an equivelar  map of type $[3^k]$ for some $k\geq 3$. The triangulation $E$ given in Fig. 1 is a  degree-regular triangulation of $\mathbb{R}^2$.


\setlength{\unitlength}{3mm}

\begin{picture}(45,16)(-5,-1.5)



\thinlines

\put(0,6){\line(1,0){34}}

\put(2,3){\line(1,0){30}} \put(2,9){\line(1,0){30}}
\put(4,12){\line(1,0){30}}

\put(0,4.5){\line(2,3){5.5}} \put(2,1.5){\line(2,3){7.5}}
\put(6,1.5){\line(2,3){7.5}} \put(10,1.5){\line(2,3){7.5}}
\put(14,1.5){\line(2,3){7.5}} \put(18,1.5){\line(2,3){7.5}}
\put(22,1.5){\line(2,3){7.5}} \put(26,1.5){\line(2,3){7.5}}
\put(30,1.5){\line(2,3){3.5}}

\put(4,1.5){\line(-2,3){3.5}} \put(8,1.5){\line(-2,3){5.5}}
\put(12,1.5){\line(-2,3){7.5}} \put(16,1.5){\line(-2,3){7.5}}
\put(20,1.5){\line(-2,3){7.5}} \put(24,1.5){\line(-2,3){7.5}}
\put(28,1.5){\line(-2,3){7.5}} \put(32,1.5){\line(-2,3){7.5}}
\put(34,4.5){\line(-2,3){5.5}} \put(34,10.5){\line(-2,3){1.5}}



\put(-1,3.4){\mbox{\small $u_{-2,-1}$}}
 \put(4.5,3.6){\mbox{\small $u_{-1,-1}$}}
 \put(11.8,3.4){\mbox{\small $u_{0,-1}$}}
 \put(15.8,3.4){\mbox{\small $u_{1,-1}$}}
 \put(19.8,3.4){\mbox{\small $u_{2,-1}$}}
 \put(23.8,3.4){\mbox{\small $u_{3,-1}$}}
 \put(27.8,3.4){\mbox{\small $u_{4,-1}$}}
 \put(31.8,3.4){\mbox{\small $u_{5,-1}$}}

 \put(-2.5,6.4){\mbox{\small $u_{-2,0}$}}
 \put(5.8,6.4){\mbox{\small $u_{-1,0}$}}
 \put(9.8,6.4){\mbox{\small $u_{0,0}$}}
 \put(13.8,6.4){\mbox{\small $u_{1,0}$}}
 \put(17.8,6.4){\mbox{\small $u_{2,0}$}}
 \put(21.8,6.4){\mbox{\small $u_{3,0}$}}
 \put(25.8,6.4){\mbox{\small $u_{4,0}$}}
 \put(29.8,6.4){\mbox{\small $u_{5,0}$}}
 \put(33.8,6.4){\mbox{\small $u_{6,0}$}}

 \put(-0.5,9.4){\mbox{\small $u_{-2,1}$}}
 \put(7.8,9.4){\mbox{\small $u_{-1,1}$}}
 \put(11.8,9.4){\mbox{\small $u_{0,1}$}}
 \put(15.8,9.4){\mbox{\small $u_{1,1}$}}
 \put(19.8,9.4){\mbox{\small $u_{2,1}$}}
 \put(23.8,9.4){\mbox{\small $u_{3,1}$}}
 \put(27.8,9.4){\mbox{\small $u_{4,1}$}}
 \put(31.8,9.4){\mbox{\small $u_{5,1}$}}

\put(5,13.2){\mbox{\small $u_{-1,2}$}}
 \put(9,13.2){\mbox{\small $u_{0,2}$}}
 \put(13,13.2){\mbox{\small $u_{1,2}$}}
 \put(17,13.2){\mbox{\small $u_{2,2}$}}
 \put(21,13.2){\mbox{\small $u_{3,2}$}}
 \put(25,13.2){\mbox{\small $u_{4,2}$}}
 \put(29,13.2){\mbox{\small $u_{5,2}$}}
 \put(33,13.2){\mbox{\small $u_{6,2}$}}


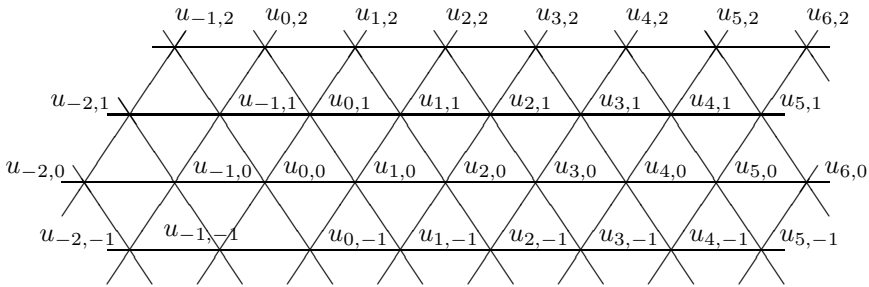
\put(5,-0.5){\mbox{{\bf Figure 1:} Regular $[3^6]$-tiling $E$ of $\mathbb{R}^2$}}

\end{picture}

From \cite{DU2005} we know

\begin{proposition}\label{prop:DU}
Let $M$ be a triangulation of the
plane $\mathbb{R}^2$. If the degree of each vertex of $M$ is $6$ then $M$
is isomorphic to $E$.
\end{proposition}

Using Proposition \ref{prop:DU}, it was shown in \cite{DU2005} that `any degree-regular
triangulation of the torus is vertex-transitive'. Here we prove

\begin{lemma} \label{lemma:face_transitive}
Let $X$ be a triangulation of the torus. If $X$ is degree-regular then the automorphism group ${\rm Aut}(X)$ acts face-transitively on $X$.
\end{lemma}

\begin{proof}
Since $X$ is degree-regular and Euler characteristic of $X$ is 0, it follows that the degree of each vertex in $X$ is 6.

Since $\mathbb{R}^2$ is the universal cover of the torus, there exists a triangulation $Y$ of $\mathbb{R}^2$ and a simplicial covering map $\eta \colon Y \to X$ (cf.
\cite[Page 144]{Sp1966}). Since the degree of each vertex in $X$ is 6, the degree of each vertex in $Y$ is 6. Because of Proposition \ref{prop:DU}, we may assume that $Y = E$. Let $\Gamma$ be the group of covering transformations. Then $|X| = |E|/\Gamma$.

We take $V= \{u_{i,2j}=(i,j\sqrt{3}), \, u_{i,2j+1} = (i+1/2, (2j+1)\sqrt{3}/2)\, : \, i, j \in \mathbb{Z}\}$ as the vertex set of $E$. Then $H := \{x\mapsto x+a, a \in V\}$ is a subgroup of ${\rm Aut}(E)$ and is called the group of translations. Clearly, $H$ is commutative.

 For $\sigma\in \Gamma$, $\eta \circ \sigma = \eta$.
So, $\sigma$ maps the geometric carrier of a simplex to the
geometric carrier of a simplex. This implies that $\sigma$
induces an automorphism $\sigma$ of $E$. Thus, we can identify
$\Gamma$ with a subgroup of ${\rm Aut}(E)$. So, $X$ is a quotient
of $E$ by the subgroup $\Gamma$ of ${\rm Aut}(E)$, where $\Gamma$
has no fixed element (vertex, edge or face). Hence $\Gamma$
consists of translations and glide reflections. Since $X =
E/\Gamma$ is orientable, $\Gamma$ does not contain any glide
reflection. Thus $\Gamma \leq H$.

Consider the subgroup $G$ of ${\rm Aut}(E)$ generated by $H$ and the map $x\mapsto -x$. So,
\begin{align*}
  G & =\{\alpha : x \mapsto \varepsilon x + a \, : \, \varepsilon = \pm 1, a \in V\} \cong H\rtimes \mathbb{Z}_2.
\end{align*}

\smallskip

\noindent {\em Claim 1.} $G$ acts face-transitively on $E$.

\smallskip

Since $H$ is vertex transitively on $E$, to prove Claim 1, it is sufficient to show that $G$ acts transitively on the set of six faces containing $u_{0,0}$. This follows from the following: $u_{-1,0}u_{0,0}u_{-1,1} + u_{1, 0} =  u_{0,0}u_{1,0}u_{0,1} =
u_{-1,-1}u_{0,-1}u_{0,0} + u_{0, 1}$, $u_{-1,0}u_{-1,-1}u_{0,0} + u_{1, 0} =  u_{0,0}u_{0,-1}u_{1,0} = u_{-1,1}u_{0,0}u_{0,1} + u_{0, -1}$ and $-1\cdot  u_{0,0}u_{-1,0}u_{-1,-1} = u_{0,0}u_{1,0}u_{0,1}$.

\medskip

\noindent {\em Claim 2.} If $K \leq H$ then $K \unlhd G$.

\smallskip

Let $\alpha \in G$ and $\beta\in K$. Assume $\alpha(x) = \varepsilon x+a$ and $\beta(x) = x + b$ for some $a, b\in V(E)$ and $\varepsilon\in\{1, -1\}$.
Then $(\alpha\circ \beta\circ\alpha^{-1})(x) = (\alpha\circ \beta)(\varepsilon(x-a)) = \alpha(\varepsilon(x-a)+b)=x-a+\varepsilon b+a= x+\varepsilon b = \beta^{\varepsilon}(x)$. Thus, $\alpha\circ \beta\circ\alpha^{-1} = \beta^{\varepsilon}\in K$. This proves Claim 2.

\smallskip

By Claim 2, $\Gamma\unlhd G$ and hence we can assume that $G/\Gamma\leq {\rm Aut}(E/\Gamma)$. Since, by Claim 1, $G$ acts face-transitively on $E$, it follows that $G/\Gamma$ acts face-transitively on $E/\Gamma$.
This completes the proof since $X = E/\Gamma$.
\end{proof}

We need the following two lemmas for the proof of Theorem \ref{theo:torus}.

\begin{lemma} \label{lemma:disc33344}
Let $X$ be a map on the $2$-disk $\mathbb{D}^2$ whose faces are triangles and quadrangles. For a vertex $x$ of $X$, let $n_3(x)$ and $n_4(x)$ be the number of triangles and quadrangles through $x$ respectively. Suppose $(n_3(u), n_4(u)) = (3, 2)$ for each internal vertex $u$. Then $X$ does not satisfy any of the following. \begin{enumerate}[{\rm (a)}]
\item $1\leq n_4(w)\leq 2$, $n_3(w)+n_4(w)\leq 4$ for one vertex $w$ on the boundary, and $(n_3(v), n_4(v))$ $= (0, 2)$ for each boundary vertex $v\neq w$.

\item $1\leq n_3(w)\leq 3$, $n_4(w)\leq 2$ and $n_3(w)+n_4(w)\leq 4$ for one vertex $w$ on the boundary, and $(n_3(v), n_4(v))= (3, 0)$ for each boundary vertex $v\neq w$.

\end{enumerate}
\end{lemma}

\begin{proof}
Let $f_0, f_1$ and $f_2$ denote the number of vertices, edges and faces of $X$ respectively. Let $n_3$ (resp., $n_4$) denote the total number of triangles (resp., quadrangles) in $X$. Let there be $n$ internal vertices and $m+1$ boundary vertices. So, $f_0=n+m+1$ and $f_2=n_3+n_4$.

Suppose $X$ satisfies (a). Then $n_4=(2n+2m+n_4(w))/4$ and $n_3= (3n+n_3(w))/3$.
Since $1\leq n_4(w) \leq 2$, it follows that $n_4(w)=2$ and hence $n_3(w)\leq 2$. These imply that $n_3(w)=0$. Thus, the exceptional vertex is like other boundary vertices. Therefore, each boundary vertex is in three edges and hence $f_1= (5n+3m+3)/2$. These imply $f_0-f_1+f_2 = (n+m+1) - (5n+3m+3)/2+ (n+(n+m+1)/2) = 0$. This is not possible since the Euler characteristic of the 2-disk $\mathbb{D}^2$ is 1.

If $X$ satisfies (b) then $n_3=(3n + 3m+ n_3(w))/3$ and $n_4= (2n+n_4(w))/4$.
Since $1\leq n_3(w) \leq 3$, it follows that $n_3(w)=3$ and hence $n_4(w)\leq 1$. These imply that $n_4(w)=0$. Thus, the exceptional vertex is like other boundary vertices and each boundary vertex is in four edges. Thus, $f_1= (5n+4m+4)/2$ and
$f_2= n_4+n_3=3n/2+m+1$. Then $f_0-f_1+f_2= (n+m+1) - (5n+4m+4)/2 + (3n/2+m+1)=0$, a contradiction again. This completes the proof.
\end{proof}

\vspace{-3mm}


\begin{figure}[ht]
\tiny
\tikzstyle{ver}=[]
\tikzstyle{vert}=[circle, draw, fill=black!100, inner sep=0pt, minimum width=4pt]
\tikzstyle{vertex}=[circle, draw, fill=black!00, inner sep=0pt, minimum width=4pt]
\tikzstyle{edge} = [draw,thick,-]
\centering

\begin{tikzpicture}[scale=0.2]
\begin{scope}[shift={(-50,13)}]

\draw[edge, thin](7.5,5)--(37.5,5);
\draw[edge, thin](7.5,10)--(37.5,10);
\draw[edge, thin](7.5,15)--(37.5,15);
\draw[edge, thin](7.5,20)--(37.5,20);
\draw[edge, thin](7.5,25)--(37.5,25);

\draw[edge, thin](10,5)--(12.5,10);
\draw[edge, thin](12.5,10)--(15,5);
\draw[edge, thin](15,5)--(17.5,10);
\draw[edge, thin](17.5,10)--(20,5);

\draw[edge, thin](20,5)--(22.5,10);
\draw[edge, thin](22.5,10)--(25,5);
\draw[edge, thin](25,5)--(27.5,10);
\draw[edge, thin](27.5,10)--(30,5);

\draw[edge, thin](30,5)--(32.5,10);
\draw[edge, thin](32.5,10)--(35,5);

\draw[edge, thin](12.5,10)--(12.5,15);
\draw[edge, thin](17.5,10)--(17.5,15);
\draw[edge, thin](22.5,10)--(22.5,15);
\draw[edge, thin](27.5,10)--(27.5,15);
\draw[edge, thin](32.5,10)--(32.5,15);

\draw[edge, thin](12.5,15)--(10,20);
\draw[edge, thin](15,20)--(12.5,15);
\draw[edge, thin](17.5,15)--(15,20);
\draw[edge, thin](20,20)--(17.5,15);
\draw[edge, thin](22.5,15)--(20,20);
\draw[edge, thin](25,20)--(22.5,15);
\draw[edge, thin](27.5,15)--(25,20);
\draw[edge, thin](30,20)--(27.5,15);
\draw[edge, thin](32.5,15)--(30,20);
\draw[edge, thin](35,20)--(32.5,15);

\draw[edge, thin](10,20)--(10,25);
\draw[edge, thin](15,20)--(15,25);
\draw[edge, thin](20,20)--(20,25);
\draw[edge, thin](25,20)--(25,25);
\draw[edge, thin](30,20)--(30,25);
\draw[edge, thin](35,20)--(35,25);

\node[ver] () at (21,4){\scriptsize $u_{-1,-1}$};
\node[ver] () at (26,4){\scriptsize $u_{0,-1}$};
\node[ver] () at (31,4){\scriptsize $u_{1,-1}$};
\node[ver] () at (35.8,4){\scriptsize $u_{2,-1}$};
\node[ver] () at (16.3,4){\scriptsize $u_{-2,-1}$};
\node[ver] () at (11.3,4){\scriptsize $u_{-3,-1}$};

\node[ver] () at (24.5,10.5){\scriptsize $u_{0,0}$};
\node[ver] () at (29.5,10.5){\scriptsize $u_{1,0}$};
\node[ver] () at (34.5,10.5){\scriptsize $u_{2,0}$};
\node[ver] () at (20,10.5){\scriptsize $u_{-1,0}$};
\node[ver] () at (15,10.5){\scriptsize $u_{-2,0}$};

\node[ver] () at (24.5,14){\scriptsize $u_{0,1}$};
\node[ver] () at (29.5,14){\scriptsize $u_{1,1}$};
\node[ver] () at (34.5,14){\scriptsize $u_{2,1}$};
\node[ver] () at (20,14){\scriptsize $u_{-1,1}$};
\node[ver] () at (15,14){\scriptsize $u_{-2,1}$};

\node[ver] () at (22,20.5){\scriptsize $u_{0,2}$};
\node[ver] () at (26.5,20.5){\scriptsize $u_{1,2}$};
\node[ver] () at (31.5,20.5){\scriptsize $u_{2,2}$};
\node[ver] () at (36.5,20.5){\scriptsize $u_{3,2}$};
\node[ver] () at (17,20.5){\scriptsize $u_{-1,2}$};
\node[ver] () at (12,20.5){\scriptsize $u_{-2,2}$};

\node[ver] () at (21,25.7){\scriptsize $u_{0,3}$};
\node[ver] () at (26,25.7){\scriptsize $u_{1,3}$};
\node[ver] () at (31,25.7){\scriptsize $u_{2,3}$};
\node[ver] () at (36,25.7){\scriptsize $u_{3,3}$};
\node[ver] () at (16,25.7){\scriptsize $u_{-1,3}$};
\node[ver] () at (10.5,25.7){\scriptsize $u_{-2,3}$};

\node[ver] () at (23, 0){\normalsize {\bf Figure 2:} Elongated triangular tiling $E_1$, };

\end{scope}
\end{tikzpicture}
\end{figure}

\vspace{-3mm}


\begin{lemma} \label{lemma:plane33344}
Let $E_1$ be the Archimedean tiling of the plane $\mathbb{R}^2$ given in Figure $2$.
If $X$ is a semi-equivelar map of $\mathbb{R}^2$ of type $[3^{3},4^{2}]$ then $X \cong E_1$.
\end{lemma}

\begin{proof}
Let the type of $X$ be $[3^{3}, 4^{2}]$. Choose a vertex $v_{0,0}$. Let the two quadrangle through $v_{0,0}$ be $v_{-1,0}\mbox{-} v_{0,0}\mbox{-} v_{0,1}\mbox{-}v_{-1,1} \mbox{-} v_{-1,0}$ and $v_{0,0}\mbox{-} v_{1,0}\mbox{-} v_{1,1}\mbox{-} v_{0,1}\mbox{-} v_{0,0}$. Then the second quadrangle through $v_{1,0}$ is of the form $v_{1,0}\mbox{-} v_{2,0}\mbox{-} v_{2,1}\mbox{-} v_{1,1}\mbox{-} v_{1,0}$ and the second quadrangle through $v_{-1,0}$ is of the form $v_{-2,0}\mbox{-} v_{-1,0}\mbox{-} v_{-1,1}\mbox{-} v_{-2,1} \mbox{-}v_{-2,0}$.
Continuing this way, we get a path $P_0 := \cdots \mbox{-} v_{-2,0}\mbox{-} v_{-1,0} \mbox{-} v_{0,0}\mbox{-}$ $v_{1,0}\mbox{-}v_{2,0}\mbox{-} \cdots$ in the edge graph of $X$ such that all the quadrangles incident with a vertex of $P_0$ lie on one side of $P_0$ and all the triangles incident with the same vertex lie on the other side of $P_0$. If $P_0$ has a closed sub-path then $P_0$ contains a cycle $W$. In that case, the bounded part of $X$ with boundary $W$ is a map on the 2-disk $\mathbb{D}^2$ which satisfies (a) or (b) of Lemma \ref{lemma:disc33344}. This is not possible by Lemma \ref{lemma:disc33344}. Thus, $P_0$ is an infinite path.
Then the faces through vertices of $P_0$ forms an infinite strip which is bounded by two infinite paths, say $P_{-1} = \cdots \mbox{-} v_{-2,-1}\mbox{-}v_{-1,-1}\mbox{-} v_{0,-1}\mbox{-} v_{1,-1}\mbox{-}v_{2,-1}\mbox{-} \cdots$ and $P_1=\cdots \mbox{-} v_{-2,1}\mbox{-}v_{-1,1}\mbox{-}v_{0,1}\mbox{-} v_{1,1}\mbox{-}v_{2,1}\mbox{-} \cdots$, where the faces between $P_0$ and $P_1$ are quadrangles and the faces  between $P_0$ and $P_{-1}$ are triangles and the faces through $v_{i,0}$ are $v_{i-1,0}\mbox{-} v_{i,0}\mbox{-} v_{i,1}\mbox{-}v_{i-1,1} \mbox{-} v_{i-1,0}$, $v_{i,0}\mbox{-} v_{i+1,0}\mbox{-} v_{i+1,1}\mbox{-} v_{i,1}\mbox{-} v_{i,0}$, $v_{i,0} v_{i+1,0} v_{i,-1}$, $v_{i,0} v_{i,-1} v_{i-1,-1}$, $v_{i,0} v_{i-1,-1} v_{i-1,0}$.

Similarly, starting with the vertex $v_{0,1}$ in place of $v_{0, 0}$ we get the paths $P_0, P_1$, $P_2=\cdots \mbox{-} v_{-2,2}\mbox{-}v_{-1,2}\mbox{-}v_{0,2}\mbox{-} v_{1,2}\mbox{-}v_{2,2}\mbox{-} \cdots$, where the faces between $P_1$ and $P_2$ are triangles and the triangles through $u_{i,1}$ are $v_{i,1} v_{i+1,1} v_{i,2}$,
$v_{i,1} v_{i,2} v_{i-1,2}$, $v_{i,1} v_{i-1,2} v_{i-1,1}$.
Continuing this way we get paths $\cdots, P_{-2}, P_{-1}, P_0, P_1, P_2, \cdots$ such that (i) the faces between $P_{2j}$ and $P_{2j+1}$ are rectangles, (ii) the faces between $P_{2j-1}$ and $P_{2j}$ are triangles, (iii) the five faces through $v_{i,2j}$ are $v_{i-1,2j} \mbox{-} v_{i,2j}\mbox{-} v_{i,2j+1}\mbox{-}v_{i-1,2j+1} \mbox{-} v_{i-1,2j}$, $v_{i,2j}\mbox{-} v_{i+1,2j}\mbox{-} v_{i+1,2j+1}\mbox{-} v_{i,2j+1}\mbox{-} v_{i,2j}$,
$v_{i,2j} v_{i+1,2j} v_{i,2j-1}$, \linebreak $v_{i,2j} v_{i,2j-1} v_{i-1,2j-1}$, \,
$v_{i,2j} v_{i-1,2j-1} v_{i-1,2j}$, and (iv) the five faces through $v_{i,2j+1}$ are \newline
$v_{i-1,2j}\mbox{-}v_{i,2j}\mbox{-}v_{i,2j+1}\mbox{-}v_{i-1,2j+1} \mbox{-} v_{i-1,2j}$, $v_{i,2j}\mbox{-} v_{i+1,2j}\mbox{-} v_{i+1,2j+1}\mbox{-} v_{i,2j+1}\mbox{-} v_{i,2j}$, $v_{i,2j+1} v_{i+1,2j+1} v_{i,2j+2}$,
$v_{i,2j+1} v_{i,2j+2} v_{i-1,2j+2}$,
$v_{i,2j+1} v_{i-1,2j+2} v_{i-1,2j+1}$
for all $j\in \mathbb{Z}$.
Then the mapping $f : V(X)\rightarrow V(E_1)$, given by $f(v_{k, t}) = u_{k,t}$ for $k, t \in \mathbb{Z}$, is an isomorphism. This proves the lemma.
\end{proof}

\begin{proof}[Proof of Theorem \ref{theo:torus}]
Let $X$ be an equivelar map of type $[6^3]$ on the torus.
Let $Y$ be the dual of $X$. Then $Y$ is an equivelar map of type $[3^6]$ on the torus and ${\rm Aut}(Y) \equiv {\rm Aut}(X)$. By Lemma \ref{lemma:face_transitive}, ${\rm Aut}(Y)$ acts face-transitively on $Y$. These imply, ${\rm Aut}(X)$ acts vertex-transitively on $X$. So, $X$ is vertex-transitive.

Now, assume that $X$ is a semi-equivelar map of type $[3^3,4^2]$ on the torus.
Since $\mathbb{R}^2$ is the
universal cover of the torus, by pulling back $X$ (using similar arguments as in the proof of Theorem 3 in \cite[Page 144]{Sp1966}), we get a semi-equivelar map $\widetilde{X}$ of type $[3^3,4^2]$ on
$\mathbb{R}^2$ and a polyhedral covering map $\eta_1 \colon \widetilde{X} \to X$.  Because of Lemma \ref{lemma:plane33344}, we may assume that $\widetilde{X} = E_1$.
Let $\Gamma_1$ be the group of covering transformations. Then $|X|
= |E_1|/\Gamma_1$.

Let $V_1$ be the vertex set of $E_1$. We take origin $(0,0)$ is the middle point of the line segment joining $u_{0,0}$ and $u_{1,1}$. Let $a = u_{1,0}-u_{0,0}$, $b= u_{0,2}-u_{0,0} \in \mathbb{R}^2$. Then $H_1 := \langle x\mapsto x+a, y\mapsto y+b \rangle$ is the group of all the translations of $E_1$. Under the action of $H_1$, vertices form two orbits.
Consider the subgroup $G_1$ of ${\rm Aut}(E_1)$ generated by $H_1$ and the map $x\mapsto -x$. So,
\begin{align*}
  G_1 & =\{\alpha : x \mapsto \varepsilon x + ma+nb \, : \, \varepsilon = \pm 1, m, n\in \mathbb{Z}\} \cong H_1\rtimes \mathbb{Z}_2.
\end{align*}
Clearly, $G_1$ acts vertex-transitively on $E_1$.

\medskip

\noindent {\em Claim.} If $K \leq H_1$ then $K \unlhd G_1$.

\smallskip

Let $g \in G_1$ and $k\in K$. Assume $g(x) = \varepsilon x+ma+nb$ and $k(x) = x + pa+ qb$ for some $m, n, p, q\in \mathbb{Z}$ and $\varepsilon\in\{1, -1\}$.
Then $(g\circ k\circ g^{-1})(x) = (g\circ k)(\varepsilon(x-ma-nb)) = g(\varepsilon(x-ma-nb)+pa+qb)=x-ma-nb+\varepsilon(pa+qb) +ma+nb= x+\varepsilon(pa+qb) = k^{\varepsilon}(x)$. Thus, $g\circ k\circ g^{-1} = k^{\varepsilon}\in K$. This proves the claim.

For $\sigma\in \Gamma_1$, $\eta_1 \circ \sigma = \eta_1$.
So, $\sigma$ maps a face of the map $E_1$ in $\mathbb{R}^2$ to a face of $E_1$ (in $\mathbb{R}^2$). This implies that $\sigma$
induces an automorphism $\sigma$ of $E_1$. Thus, we can identify
$\Gamma_1$ with a subgroup of ${\rm Aut}(E_1)$. So, $X$ is a quotient
of $E_1$ by the subgroup $\Gamma_1$ of ${\rm Aut}(E_1)$, where $\Gamma_1$
has no fixed element (vertex, edge or face). Hence $\Gamma_1$
consists of translations and glide reflections. Since $X =
E_1/\Gamma_1$ is orientable, $\Gamma_1$ does not contain any glide
reflection. Thus $\Gamma_1 \leq H_1$. By the claim, $\Gamma_1$ is a normal subgroup of $G_1$. Since $G_1$ acts transitively on $V_1$, $G_1/\Gamma_1$ acts
transitively on the vertices of $E_1/\Gamma_1$. Thus, $X$ is vertex-transitive.
\end{proof}

\newpage

\section{Examples of maps on the torus and Klein bottle} \label{sec:examples}

\begin{example} \label{eg:8maps-torus}
{\rm Eight types of semi-equivelar maps on the torus given in Figure 3.}

\begin{figure}[ht!]
\tiny
\tikzstyle{ver}=[]
\tikzstyle{vert}=[circle, draw, fill=black!100, inner sep=0pt, minimum width=4pt]
\tikzstyle{vertex}=[circle, draw, fill=black!00, inner sep=0pt, minimum width=4pt]
\tikzstyle{edge} = [draw,thick,-]

\end{figure}

\vspace{-8mm}

\begin{figure}[ht!]
\tiny
\tikzstyle{ver}=[]
\tikzstyle{vert}=[circle, draw, fill=black!100, inner sep=0pt, minimum width=4pt]
\tikzstyle{vertex}=[circle, draw, fill=black!00, inner sep=0pt, minimum width=4pt]
\tikzstyle{edge} = [draw,thick,-]
%
\end{figure}

\vspace{-9mm}

\begin{figure}[ht!]
\tiny
\tikzstyle{ver}=[]
\tikzstyle{vert}=[circle, draw, fill=black!100, inner sep=0pt, minimum width=4pt]
\tikzstyle{vertex}=[circle, draw, fill=black!00, inner sep=0pt, minimum width=4pt]
\tikzstyle{edge} = [draw,thick,-]
%
\end{figure}

\vspace{-9mm}

\begin{figure}[ht!]
\tiny
\tikzstyle{ver}=[]
\tikzstyle{vert}=[circle, draw, fill=black!100, inner sep=0pt, minimum width=4pt]
\tikzstyle{vertex}=[circle, draw, fill=black!00, inner sep=0pt, minimum width=4pt]
\tikzstyle{edge} = [draw,thick,-]
%
\end{figure}

{\rm It follows from Theorem \ref{theo:torus} that the map $T_1$  is vertex-transitive. }
\end{example}


\newpage

\begin{example} \label{eg:10maps-Kbottle}
{\rm Ten types of semi-equivelar maps on the Klein bottle given in Figure 4.}
\end{example}

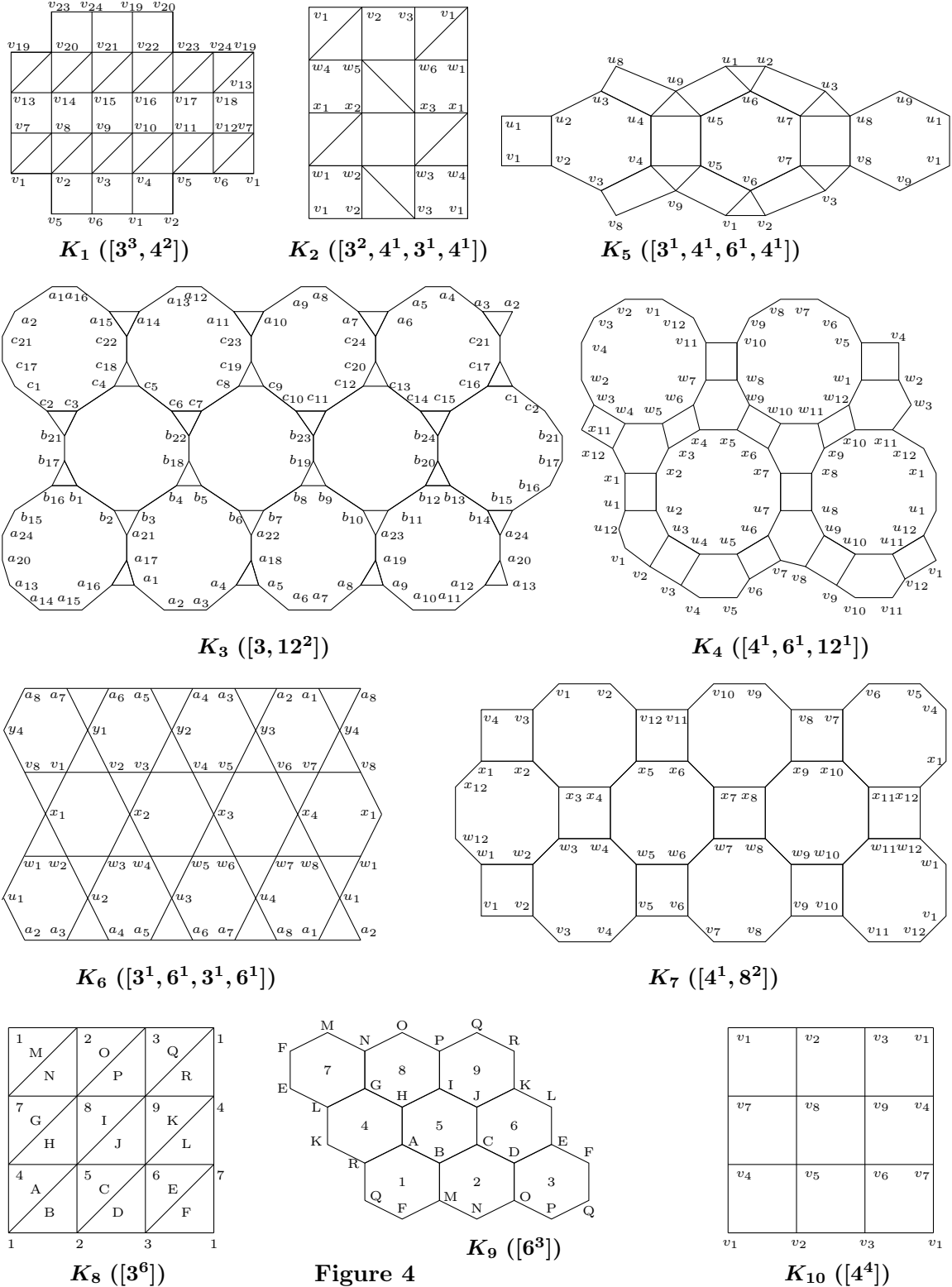
\begin{figure}[ht!]
\tiny
\tikzstyle{ver}=[]
\tikzstyle{vert}=[circle, draw, fill=black!100, inner sep=0pt, minimum width=4pt]
\tikzstyle{vertex}=[circle, draw, fill=black!00, inner sep=0pt, minimum width=4pt]
\tikzstyle{edge} = [draw,thick,-]

%
\end{figure}

\vspace{-5mm}

\begin{figure}[ht!]
\tiny
\tikzstyle{ver}=[]
\tikzstyle{vert}=[circle, draw, fill=black!100, inner sep=0pt, minimum width=4pt]
\tikzstyle{vertex}=[circle, draw, fill=black!00, inner sep=0pt, minimum width=4pt]
\tikzstyle{edge} = [draw,thick,-]
%

\end{figure}

\vspace{-5mm}


\begin{figure}[ht!]
\tiny
\tikzstyle{ver}=[]
\tikzstyle{vert}=[circle, draw, fill=black!100, inner sep=0pt, minimum width=4pt]
\tikzstyle{vertex}=[circle, draw, fill=black!00, inner sep=0pt, minimum width=4pt]
\tikzstyle{edge} = [draw,thick,-]
\centering
%

\end{figure}

\vspace{-5mm}

\begin{figure}[ht!]
\tiny
\tikzstyle{ver}=[]
\tikzstyle{vert}=[circle, draw, fill=black!100, inner sep=0pt, minimum width=4pt]
\tikzstyle{vertex}=[circle, draw, fill=black!00, inner sep=0pt, minimum width=4pt]
\tikzstyle{edge} = [draw,thick,-]

%


\end{figure}

In the next two proofs, we denote the $n$-cycle whose edges are $u_{1}u_{2}, \dots, u_{n-1}u_{n}, u_{n}u_{1}$ by $C_n(u_1, \dots, u_n)$. This helps us to compare different sizes of cycles.

\begin{lemma} \label{lemma:torus-se-nonwr}
The semi-equivelar maps $T_2, \dots, T_8$ in Example $\ref{eg:8maps-torus}$ are not vertex-transitive.
\end{lemma}

\begin{proof}
Let $\mathcal{G}_2$ be the graph whose vertices are the vertices of $T_2$ and edges are the diagonals of $4$-gons of $T_2$. Then $\mathcal{G}_2$ is a $2$-regular graph. Hence, $\mathcal{G}_2$ is a disjoint union of cycles. Clearly, ${\rm Aut}(T_2)$ acts on $\mathcal{G}_2$. If the action of ${\rm Aut}(T_2)$ is vertex-transitive on $T_2$ then it would be vertex-transitive on $\mathcal{G}_2$. But this is not possible since $C_4(u_1, u_4, u_{8}, u_{11})$, $C_{12}(v_{1}, v_4, v_{9}, v_{12}, v_3, v_{6}, v_{8}, v_{11}, v_2, v_5, v_{7}, v_{10})$ are components of $\mathcal{G}_2$ of different sizes.

Let $\mathcal{G}_3$ be the graph whose vertices are the vertices of $T_3$ and edges are the long diagonals of $12$-gons of $T_3$. Then $\mathcal{G}_3$ is a $2$-regular graph. Hence, $\mathcal{G}_3$ is a disjoint union of cycles. Clearly, ${\rm Aut}(T_3)$ acts on $\mathcal{G}_3$. If the action of ${\rm Aut}(T_3)$ is vertex-transitive on $T_3$ then it would be vertex-transitive on $\mathcal{G}_3$. But this is not possible since $C_4(a_{17}, a_{22}, a_{19}, a_{24})$ and $C_{12}(c_{1},$ $a_6, b_{9}, c_{14}, a_1, b_6, c_9, a_{14}, b_1, c_6, a_9, b_{14})$ are components of $\mathcal{G}_3$ of different sizes.

Let $\mathcal{G}_4$ be the graph whose vertices are the vertices of $T_4$ and edges are the diagonals of $4$-gons and long diagonals of $12$-gons of $T_4$. Then $\mathcal{G}_4$ is a $2$-regular graph. Clearly, ${\rm Aut}(T_4)$ acts on $\mathcal{G}_4$. If the action of ${\rm Aut}(T_4)$ is vertex-transitive on $T_4$ then it would be vertex-transitive on $\mathcal{G}_4$. But this is not possible since $C_8(v_{2},$ $u_{4},$ $x_{5},$ $w_{10},$ $v_8,$ $u_{10},$ $x_{11},$ $w_4)$ and
$C_4(x_{1},$ $u_{2},$ $x_7,$ $u_8)$ are components of $\mathcal{G}_4$ of different sizes.

Let $\mathcal{G}_5$ be the graph whose vertices are the vertices of $T_5$ and edges are the diagonals of $4$-gons of $T_5$. Then $\mathcal{G}_5$ is a $2$-regular graph. Hence, $\mathcal{G}_5$ is a disjoint union of cycles. Clearly, ${\rm Aut}(T_5)$ acts on $\mathcal{G}_5$. If the action of ${\rm Aut}(T_5)$ is vertex-transitive on $T_5$ then it would be vertex-transitive on $\mathcal{G}_5$. But this is not possible since $C_6(x_9,$ $v_3,$ $x_{3},$ $v_6,$ $x_{6},$ $v_{9})$ and $C_4(u_{2},$ $v_2,$ $x_{1},$ $w_2)$ are components of $\mathcal{G}_5$ of different sizes.

Let $\mathcal{G}_6$ be the graph whose vertices are the vertices of $T_6$ and edges are the long diagonals of $6$-gons of $T_6$. Then $\mathcal{G}_6$ is a $2$-regular graph. Hence, $\mathcal{G}_6$ is a disjoint union of cycles. Clearly, ${\rm Aut}(T_6)$ acts on $\mathcal{G}_6$. If the action of ${\rm Aut}(T_6)$ is vertex-transitive on $T_6$ then it would be vertex-transitive on $\mathcal{G}_6$. But this is not possible since $C_8(w_{1},$ $w_{2},$ $w_7,$ $w_8,$ $w_5,$ $w_6,$ $w_3,$ $w_4)$ and $C_4(u_{1},$ $u_{2},$ $u_3,$ $u_{4})$ are components of $\mathcal{G}_6$ of different sizes.

Let $\mathcal{G}_7$ be the graph whose vertices are the vertices of $T_7$ and edges are the diagonals of $4$-gons and common edges between any two $8$-gons of $T_7$. Then $\mathcal{G}_7$ is a $2$-regular graph. Hence, $\mathcal{G}_7$ is a disjoint union of cycles. Clearly, ${\rm Aut}(T_7)$ acts on $\mathcal{G}_7$. If the action of ${\rm Aut}(T_7)$ is vertex-transitive on $T_7$ then it would be vertex-transitive on $\mathcal{G}_7$. But this is not possible since $C_8(v_{1},$ $w_{2},$ $w_3,$ $x_{4},$ $x_5,$ $u_{6},$ $u_7,$ $v_{12})$ and $C_{24}(v_{2},$ $w_{1},$ $w_{12},$ $x_{11}, x_{10}, u_9, u_8, v_{11}, v_{10}, w_9, w_8, x_7, x_6, u_5, u_4,
v_7, v_6, w_5, w_4, x_3, x_2, u_1, u_{12}, v_{3})$ are components of $\mathcal{G}_7$ of different sizes.

We call an edge $uv$ of $T_8$ {\em nice} if at $u$ (respectively, at $v$) three 3-gons containing $u$ (respectively, $v$) lie on one side of $uv$ and one on the other side of $uv$. (For example, $v_{10}v_{15}$ is nice). Observe that there is exactly one nice edge in $T_8$ through each vertex.
Let $\mathcal{G}_8$ be the graph whose vertices are the vertices of $T_8$ and edges are the nice edges and the long diagonals of $6$-gons. Then $\mathcal{G}_8$ is a $2$-regular graph. Hence, $\mathcal{G}_8$ is a disjoint union of cycles. Clearly, ${\rm Aut}(T_8)$ acts on $\mathcal{G}_8$. If the action of ${\rm Aut}(T_8)$ is vertex-transitive on $T_8$ then it would be vertex-transitive on $\mathcal{G}_8$. But this is not possible since $C_4(v_{7},$ $v_{15},$ $v_{10},$ $v_{18})$ and $C_8(v_{1},$ $v_{23},$ $v_{17},$ $v_{11},$ $v_4,$ $v_{20},$ $v_{14}$ $v_8)$ are components of $\mathcal{G}_8$ of different sizes.
\end{proof}

\begin{proof}[Proof of Theorem \ref{theo:non-wr-torus}]
The result follows from Lemma \ref{lemma:torus-se-nonwr}.
\end{proof}

\begin{lemma}\label{lemma:Kbottle-se-nonwr}
The maps $K_{1},$ \dots, $K_{10}$ are not vertex-transitive.
\end{lemma}

\begin{proof}
Let $\mathcal{H}_{1}$ be the graph whose vertices are the vertices of $K_{1}$ and edges are the diagonals of $4$-gons of $K_{1}$. Then $\mathcal{H}_{1}$ is a $2$-regular graph. Hence, $\mathcal{H}_{1}$ is a disjoint union of cycles. Clearly, ${\rm Aut}(K_{1})$ acts on $\mathcal{H}_{1}$. If the action of ${\rm Aut}(K_{1})$ is vertex-transitive on $K_{1}$ then it would be vertex-transitive on $\mathcal{H}_{1}$. But this is not possible since $C_6(v_{7}, v_{14}, v_9, v_{16}, v_{11}, v_{18})$ and $C_3(v_{20}, v_{24}, v_{22})$ are two components of $\mathcal{H}_{1}$ of different sizes.

There are exactly two induced $3$-cycles in $K_{2}$, namely, $C_3(x_{1}, x_{2}, x_3)$ and $C_3(v_{1}, v_{2}, v_{3})$. So, some vertices of $K_2$ are in an induced 3-cycle and some are not. Therefore, the action of ${\rm Aut}(K_{2})$ on $K_2$ can not be  vertex-transitive.

Like $\mathcal{G}_3$ in the proof of Lemma \ref{lemma:torus-se-nonwr}, let $\mathcal{H}_{3}$ be the graph whose vertices are the vertices of $K_{3}$ and edges are the long diagonals of $12$-gons of $K_{3}$.
Then, ${\rm Aut}(K_{3})$ acts on the 2-regular graph $\mathcal{H}_{3}$. If the action of ${\rm Aut}(K_{3})$ is vertex-transitive on $K_{3}$ then it would be vertex-transitive on $\mathcal{H}_{3}$. But this is not possible since $C_4(a_{17}, a_{22}, a_{19}, a_{24})$ and $C_{24}(a_3, b_4, c_3, a_1, b_6, c_9, a_7, b_8, c_7, a_{13}, b_2, c_5, a_{11}, b_{12}, c_{11}, a_9, b_{14}, c_1, a_{15}, b_{16}, c_{15}, a_5, b_{10}, c_{13})$ are components of $\mathcal{H}_{3}$ of different sizes.

Let $\mathcal{H}_{4}$ be the graph whose vertices are the vertices of $K_{4}$ and edges are the diagonals of $4$-gons and long diagonals of $12$-gons of $K_{4}$ (like $\mathcal{G}_4$ in the proof of Lemma \ref{lemma:torus-se-nonwr}).
Then, ${\rm Aut}(K_{4})$ acts on the 2-regular graph $\mathcal{H}_{4}$. If the action of ${\rm Aut}(K_{4})$ is vertex-transitive on $K_{4}$ then it would be vertex-transitive on $\mathcal{H}_{4}$. But this is not possible since $C_4(v_{5}, w_{2}, v_{11}, w_{8})$ and $C_8(v_{2}, u_{4}, x_5, w_{10}, v_7, u_5, x_4, w_5)$ are components of $\mathcal{H}_{4}$ of different sizes.

Let $\mathcal{H}_{5}$ be the graph whose vertices are the vertices of $K_{5}$ and edges are the diagonals of $4$-gons in $K_{5}$ (like $\mathcal{G}_5$). Then, ${\rm Aut}(K_{5})$ acts on the 2-regular graph $\mathcal{H}_{5}$. If the action of ${\rm Aut}(K_{5})$ is vertex-transitive on $K_{5}$ then it would be vertex-transitive on $\mathcal{H}_{5}$. But this is not possible since $C_{12}(v_1, u_2, u_7, v_8, v_4, u_5, u_1, v_2,$ $v_7, u_8, u_4, v_5)$ and $C_3(u_3, u_9, u_6)$ are components of $\mathcal{H}_{5}$ of different sizes.

Let $\mathcal{H}_{6}$ be the graph whose vertices are the vertices of $K_{6}$ and edges are the long diagonals of $6$-gons of $K_{6}$ (like $\mathcal{G}_6$). Then, ${\rm Aut}(K_{6})$ acts on the 2-regular graph $\mathcal{H}_{6}$. If the action of ${\rm Aut}(K_{6})$ is vertex-transitive on $K_{6}$ then it would be vertex-transitive on $\mathcal{H}_{6}$. But this is not possible since $C_{24}(a_{2}, w_{2}, v_2, a_{5}, w_{3}, v_{1}$, $a_8,$ $w_8$, $v_8, a_7, w_5, v_3, a_6, w_6, v_6, a_1, w_7, v_5, a_4$, $w_4$, $v_4$, $a_3, w_1, v_7)$ and $C_4(u_{1}, u_{2}, u_{3}, u_4)$ are components of $\mathcal{H}_{6}$ of different sizes.

Let $\mathcal{H}_{7}$ be the graph whose vertices are the vertices of $K_{7}$ and edges are the diagonals of $4$-gons and common edges between any two $8$-gons in $K_{7}$ (like $\mathcal{G}_7$). Then ${\rm Aut}(K_{7})$ acts on the 2-regular graph $\mathcal{H}_{7}$. If the action of ${\rm Aut}(K_{7})$ is vertex-transitive on $K_{7}$ then it would be vertex-transitive on $\mathcal{H}_{7}$. But this is not possible since $C_{24}(v_1, w_2, w_3, x_4, x_5, v_{11}, v_{10}$, $w_{9}$, $w_{8}$, $x_7$, $x_6, v_{12}, v_2, w_1, w_{12}, x_{11}, x_{10}, v_8$, $v_9,$ $w_{10}$, $w_{11}, x_{12}, x_1, v_3)$ and $C_{12}(v_5, w_6, w_7, x_8, x_9$, $v_7$, $v_6$, $w_5$, $w_4, x_3, x_2, v_4)$ are components of $\mathcal{H}_{7}$ of different sizes.

Let ${\rm Skel}_1(K_8)$ be the edge graph of $K_8$ and $\mathcal{N}_8$ be the non-edge graph (i.e., the complement of ${\rm Skel}_1(K_8)$) of $K_8$. If ${\rm Aut}(K_8)$ acts vertex-transitively then ${\rm Aut}(K_8)$ acts vertex-transitively on ${\rm Skel}_1(K_8)$ and hence on $\mathcal{N}_8$. But, this is not possible since $\mathcal{N}_8$ is the union of two cycles of different lengths, namely, $\mathcal{N}_8 = C_6(2, 4, 3, 5, 7, 9) \sqcup C_3(1, 6, 8)$.

Consider the triangles $C=256$ and $O=238$ in $K_8$. If there exists $\alpha\in {\rm Aut}(K_8)$ such that $\alpha(C) = O$ then $\alpha$ acts on $\mathcal{N}_8 = C_6(2, 4, 3, 5, 7, 9) \sqcup C_3(1, 6, 8)$ and hence $\alpha(6) = 8$, $\alpha(\{2, 5\}) = \{2, 3\}$. This is not possible, since $25$ is a long diagonal in $C_6(2, 4, 3, 5, 7, 9)$ where as $23$ is a short diagonal in $C_6(2, 4, 3, 5, 7, 9)$. Thus, the action of ${\rm Aut}(K_8)$ on $K_8$ is not face-transitive. Observe that $K_9$ is the dual of $K_8$. Hence the action of ${\rm Aut}(K_9) = {\rm Aut}(K_8)$ on $K_9$ is not vertex-transitive.

There are exactly four induced $3$-cycles in $K_{10}$, namely, $C_3(v_1, v_2, v_3)$, $C_3(v_1, v_4, v_7)$, $C_3(v_2, v_5, v_8)$ and $C_3(v_3, v_6, v_9)$. Let $\mathcal{H}_{10} := C_3(v_1, v_2, v_3) \cup C_3(v_1, v_4, v_7) \cup C_3(v_2, v_5, v_8) \cup C_3(v_3, v_6, v_9)$. Clearly, ${\rm Aut}(K_{10})$ acts on $\mathcal{H}_{10}$. If the action of ${\rm Aut}(K_{10})$ is vertex-transitive on $K_{10}$ then it would be vertex-transitive on $\mathcal{H}_{10}$. But this is not possible since the degrees of all the vertices in $\mathcal{H}_{10}$ are not same.
\end{proof}

\begin{proof}[Proof of Theorem \ref{theo:non-wr-kbottle}]
The result follows from Lemma \ref{lemma:Kbottle-se-nonwr}.
\end{proof}


\section{Proof of Theorem \ref{theo:torus33434&3636}}

\begin{lemma} \label{lemma:disc33434}
Let $X$ be a map on the $2$-disk $\mathbb{D}^2$ whose faces are triangles and quadrangles. For a vertex $x$ of $X$, let $n_3(x)$ and $n_4(x)$ be the number of triangles and quadrangles through $x$ respectively. Then $X$ does not satisfy all the four properties. {\rm (i)} $(n_3(u), n_4(u)) = (3, 2)$ for each internal vertex $u$, {\rm (ii)} $n_3(w)\leq 3$, $n_4(w)\leq 2$, $n_3(w)+n_4(w)\leq 4$, $(n_3(w), n_4(w))\neq (3,0)$, $(0,2)$ for one vertex $w$ on the boundary, {\rm (iii)} $(n_3(v), n_4(v))= (1, 1)$ or $(2,1)$ for each boundary vertex $v\neq w$, and {\rm (iv)} $n_3(v_1)+n_3(v_2)=3$ for each boundary edge $v_1v_2$ not containing $w$.
\end{lemma}

\begin{proof}
Let $f_0, f_1$ and $f_2$ denote the number of vertices, edges and faces of $X$ respectively. Let $n_3$ (resp., $n_4$) denote the total number of triangles (resp., quadrangles) in $X$. Let there be $n$ internal vertices and $m+1$ boundary vertices. So, $f_0=n+m+1$ and $f_2=n_3+n_4$.

Suppose $X$ satisfies (i), (ii), (iii) and (iv).
First assume that $m$ is even. Let $m=2p$. Then $n_3=(3n+2p+p+n_3(w))/3$ and $n_4=(2n+2p+n_4(w))/4$. So, $n_3(w)\in\{0,3\}$ and $n_4(w)\in \{0, 2\}$. Since $1\leq n_3(w)+n_4(w)\leq 4$, these imply $(n_3(w), n_4(w)) \in \{(3,0), (0,2)\}$, a contradiction. So, $m$ is odd. Let $m=2q+1$. Then $n_4=(2n+2q+1+n_4(w))/4$. So, $n_4(w) =1$. Now, $n_3=(3n+2q+q+\varepsilon+n_3(w))/3$, where $\varepsilon=1$ or 2 depending on whether the number of boundary vertices which are in one triangle is $q+1$ or $q$. So, $\varepsilon+n_3(w) = 3$. This implies that the alternate vertices on the boundary are in 1 and 2 triangles and the degrees of $q+1$ boundary vertices are 4 and the degrees of the other $q+1$ vertices are 3.
Thus, $f_2= (n+q+1)/2+(n+q+1)$ and $f_1= (5n+4(q+1)+3(q+1))/2$. Then $f_0-f_1+f_2 = (n+2q+2) -(5n+7q+7)/2 + (3n+3q+3)/2 = 0$. This is not possible since the Euler characteristic of the 2-disk $\mathbb{D}^2$ is 1. This completes the proof.
\end{proof}

\begin{lemma} \label{lemma:disc3636}
Let $X$ be a map on the $2$-disk $\mathbb{D}^2$ whose faces are triangles and hexagons. For a vertex $x$ of $X$, let $n_3(x)$ and $n_6(x)$ be the number of triangles and hexagons  through $x$ respectively. Then $X$ does not satisfy all the three properties. {\rm (i)} $(n_3(u), n_6(u)) = (2, 2)$ for each internal vertex $u$, {\rm (ii)} $n_3(w), n_6(w)\leq 2$, $1\leq n_3(w)+n_6(w)\leq 3$, for one vertex $w$ on the boundary, and {\rm (iii)} $(n_3(v), n_6(v))= (1, 1)$ for each boundary vertex $v\neq w$.
\end{lemma}

\begin{proof}
Let $f_0, f_1$ and $f_2$ denote the number of vertices, edges and faces of $X$ respectively. Let $n_3$ (resp., $n_6$) denote the total number of triangles (resp., hexagons) in $X$. Let there be $n$ internal vertices and $m+1$ boundary vertices. So, $f_0=n+m+1$ and $f_2=n_3+n_6$.

Suppose $X$ satisfies (i), (ii) and (iii).
Then $n_3=(2n+m+n_3(w))/3$ and $n_6=(2n+m+n_6(w))/6$. So, $n_6(w) - n_3(w)= 6 n_6-3n_3 = 3(2n_6-n_3)$. Since $0\leq n_3(w), n_6(w)\leq 2$, these imply $n_6(w) - n_3(w) =0$. So, $n_6(w) = n_3(w)$. Since $1\leq n_3(w)+n_4(w)\leq 3$, these imply that $n_6(w) = n_3(w)=1$.
Thus, the exceptional vertex is like other boundary vertices. Therefore, each boundary vertex is in three edges and hence $f_1= (4n+3(m+1))/2$. So, $m+1$ is even, say $m+1=2\ell$.
Thus, $f_1=2n+3\ell$.
Now, since $n_6(w) = n_3(w)=1$, $f_2=n_3+n_6= (2n+m+1)/3+(2n+m+1)/6=(2n+m+1)/2= n+\ell$.
Then $f_0-f_1+f_2= (n+2\ell)-(2n+3\ell)+(n+\ell)= 0$. This is not possible since the Euler characteristic of the 2-disk $\mathbb{D}^2$ is 1. This completes the proof.
\end{proof}

\begin{lemma} \label{lemma:plane33434&3636}
Let $E_2$ and $E_6$ be the Archimedean tilings of $\mathbb{R}^2$ given in Figure $5$.
Let $Y$ be a semi-equivelar map on the plane $\mathbb{R}^2$. {\rm (a)} If the type of $Y$ is $[3^{2}, 4^1, 3^1, 4^1]$ then $Y \cong E_2$. {\rm (b)} If the type of $Y$ is $[3^{1}, 6^1, 3^1, 6^1]$ then $Y \cong E_6$.
\end{lemma}

\begin{proof}
If the type of $Y$ is $[3^{2}, 4^1, 3^1, 4^1]$ then by the similar arguments as in the proof of Lemma \ref{lemma:plane33344}, we get $Y \cong E_2$. In this case, to show that the path in $Y$ (similar to the path $P_0$ in the proof of Lemma \ref{lemma:plane33344})
corresponding to the path $\cdots \mbox{-} u_{-2,0} \mbox{-} u_{-1,0} \mbox{-} u_{0,0} \mbox{-}u_{1,0} \mbox{-} u_{2,0} \mbox{-} u_{3,0} \mbox{-}\cdots$ in $E_2$
 is an infinite path, we need to use that there is no map on the 2-disk $\mathbb{D}^2$ which satisfies (i) - (iv) of Lemma \ref{lemma:disc33434}.

If the type of $Y$ is $[3^{1}, 6^1, 3^1, 6^1]$ then by the similar arguments as in the proof of Lemma \ref{lemma:plane33344}, we get $Y \cong E_6$. In this case, to show that the path
in $Y$ corresponding to the path $\cdots \mbox{-} v_{-2,0} \mbox{-}w_{-2,0} \mbox{-} v_{-1,0} \mbox{-} w_{-1,0} \mbox{-} v_{0,0} \mbox{-} w_{0,0}\mbox{-}v_{1,0} \mbox{-}w_{1,0} \mbox{-}v_{2,0} \mbox{-} w_{2,0} \mbox{-} \cdots$ in $E_6$
is an infinite path, we need to use that there is no map on the 2-disk $\mathbb{D}^2$ which satisfies (i) - (iii) of Lemma \ref{lemma:disc3636}.
\end{proof}



\begin{figure}[ht]
\tiny
\tikzstyle{ver}=[]
\tikzstyle{vert}=[circle, draw, fill=black!100, inner sep=0pt, minimum width=4pt]
\tikzstyle{vertex}=[circle, draw, fill=black!00, inner sep=0pt, minimum width=4pt]
\tikzstyle{edge} = [draw,thick,-]

\begin{tikzpicture}[scale=0.2]
\begin{scope}[shift={(-15,-7)}]

\draw[edge, thin](5,5)--(10,5)--(10,10)--(5,10)--(5,5);
\draw[edge, thin](15,2.5)--(20,2.5)--(20,7.5)--(15,7.5)--(15,2.5);
\draw[edge, thin](25,0)--(30,0)--(30,5)--(25,5)--(25,0);
\draw[edge, thin](10,5)--(15,2.5);
\draw[edge, thin](10,10)--(15,7.5);
\draw[edge, thin](10,5)--(15,7.5);
\draw[edge](4,5.5)--(5,5)--(10,5)--(15,2.5)--(20,2.5)--(25,0)--(30,0)--(31,-0.5);

\draw[edge](12,-8.5)--(12.5,-7.5)--(12.5,-2.5)--(15,2.5)--(15,7.5)--(17.5,12.5)--(17.5,17.5)--(18,18.5);

\node[ver] () at (10.3,-5.2){\scriptsize $u_{-1,-2}$};
\node[ver] () at (15,-7){\scriptsize $u_{0,-2}$};
\node[ver] () at (20.2,-7.7){\scriptsize $u_{1,-2}$};

\node[ver] () at (10.2,0){\scriptsize $u_{-1,-1}$};
\node[ver] () at (15.1,-3.5){\scriptsize $u_{0,-1}$};
\node[ver] () at (19.8,-2.7){\scriptsize $u_{1,-1}$};
\node[ver] () at (24.9,-4.5){\scriptsize $u_{2,-1}$};
\node[ver] () at (29.7,-5){\scriptsize $u_{3,-1}$};

\node[ver] () at (7,5.5){\scriptsize $u_{-2,0}$};
\node[ver] () at (13,4.8){\scriptsize $u_{-1,0}$};
\node[ver] () at (13,1.8){\scriptsize $u_{0,0}$};
\node[ver] () at (22.5,2.3){\scriptsize $u_{1,0}$};
\node[ver] () at (26.7,.5){\scriptsize $u_{2,0}$};
\node[ver] () at (32,-.2){\scriptsize $u_{3,0}$};

\node[ver] () at (7.5,10.5){\scriptsize $u_{-2,1}$};
\node[ver] () at (12.5,10){\scriptsize $u_{-1,1}$};
\node[ver] () at (17,8){\scriptsize $u_{0,1}$};
\node[ver] () at (22,7.3){\scriptsize $u_{1,1}$};
\node[ver] () at (27,5.5){\scriptsize $u_{2,1}$};
\node[ver] () at (32,4.8){\scriptsize $u_{3,1}$};

\node[ver] () at (10,14.2){\scriptsize $u_{-2,2}$};
\node[ver] () at (15.5,14.6){\scriptsize $u_{-1,2}$};
\node[ver] () at (19.6,11.7){\scriptsize $u_{0,2}$};
\node[ver] () at (25,12.2){\scriptsize $u_{1,2}$};
\node[ver] () at (29.1,10.6){\scriptsize $u_{2,2}$};

\node[ver] () at (19.2,16.7){\scriptsize $u_{0,3}$};
\node[ver] () at (24.7,17.3){\scriptsize $u_{1,3}$};
\node[ver] () at (29.1,14){\scriptsize $u_{2,3}$};

 \draw[edge, thin](20,2.5)--(25,0);
 \draw[edge, thin](20,7.5)--(25,5);
 \draw[edge, thin](20,2.5)--(25,5);
 \draw[edge, thin](30,5)--(31,7);
 \draw[edge, thin](30,5)--(31,4.5);
 \draw[edge, thin](30,0)--(31,0.5);
 \draw[edge, thin](5,10)--(4,10.4);

 \draw[edge, thin](10,10)--(7.5,15)--(5,10);
 \draw[edge, thin](10,10)--(7.5,15)--(12.5,15)--(10,10);
 \draw[edge, thin](20,7.5)--(17.5,12.5)--(15,7.5);
 \draw[edge, thin](20,7.5)--(17.5,12.5)--(22.5,12.5)--(20,7.5);
 \draw[edge, thin](30,5)--(27.5,10)--(25,5);
 \draw[edge, thin](12.5,15)--(17.5,12.5);
 \draw[edge, thin](22.5,12.5)--(27.5,10);
 \draw[edge, thin](17.5,17.5)--(12.5,15);
\draw[edge, thin](17.5,12.5)--(22.5,12.5)--(22.5,17.5)--(17.5,17.5)--(17.5,12.5);
 \draw[edge, thin](22.5,17.5)--(27.5,15)--(22.5,12.5);
 \draw[edge, thin](27.5,15)--(27.5,10);
 \draw[edge, thin](27.5,15)--(31,15);
 \draw[edge, thin](27.5,10)--(31,10);
 \draw[edge, thin](7.5,15)--(5.5,16);
 \draw[edge, thin](12.5,15)--(12.5,18);
 \draw[edge, thin](17.5,17.5)--(15.5,18.5);
 \draw[edge, thin](5,10)--(4,9.5);
 \draw[edge, thin](5,5)--(3.5,3);
 \draw[edge, thin](7.5,15)--(7.5,18);
 \draw[edge, thin](22.5,17.5)--(23.5,19);
 \draw[edge, thin](22.5,17.5)--(21.5,19);
 \draw[edge, thin](27.5,15)--(28.5,17);

 \draw[edge, thin](10,5)--(7.5,0)--(5,5);
 \draw[edge, thin](10,5)--(7.5,0)--(12.5,-2.5)--(15,2.5);
 \draw[edge, thin](20,2.5)--(17.5,-2.5)--(15,2.5);
 \draw[edge, thin](17.5,-2.5)--(12.5,-2.5);
 \draw[edge, thin](30,0)--(27.5,-5)--(25,0);
 \draw[edge, thin](27.5,-5)--(22.5,-5)--(25,0);
 \draw[edge, thin](22.5,-5)--(17.5,-2.5);
\draw[edge, thin](17.5,-2.5)--(12.5,-2.5)--(12.5,-7.5)--(17.5,-7.5)--(17.5,-2.5);
 \draw[edge, thin](22.5,-5)--(17.5,-7.5);
 \draw[edge, thin](12.5,-2.5)--(12.5,-7.5)--(7.5,-5)--(12.5,-2.5);
 \draw[edge, thin](7.5,-5)--(7.5,0);
 \draw[edge, thin](7.5,-5)--(6.5,-7);
 \draw[edge, thin](7.5,-5)--(4.5,-5);
 \draw[edge, thin](7.5,0)--(4.5,0);
 \draw[edge, thin](22.5,-5)--(22.5,-7.5);
 \draw[edge, thin](17.5,-7.5)--(19.5,-8.5);
 \draw[edge, thin](17.5,-7.5)--(17,-8.5);
 \draw[edge, thin](12.5,-7.5)--(13,-8.5);
 \draw[edge, thin](27.5,-5)--(27.5,-7);
 \draw[edge, thin](27.5,-5)--(29.5,-6);

 \node[ver] () at (17.5,5){\scriptsize $\bullet$};
\node[ver] () at (27.5,2.5){\scriptsize $\bullet$};
\node[ver] () at (7.5,7.5){\scriptsize $\bullet$};

\node[ver] () at (15,-5){\scriptsize $\bullet$};
\node[ver] () at (20,15){\scriptsize $\bullet$};
 \draw [dashed] (3.5,8.5) -- (31.5,1.5);
\draw [dashed] (14,-9) -- (21,19);

\node[ver] () at (15,-12){\small (a): Snub square tiling $E_2$};

\node[ver] () at (15,-14){\mbox{}};

\end{scope}
\end{tikzpicture}
\begin{tikzpicture}[scale=.15]

\draw[edge, thin](-1.5,0)--(51.5,0);
\draw[edge, thin](-1.5,10)--(51.5,10);
\draw[edge, thin](-1.5,20)--(51.5,20);
\draw[edge, thin](-1.5,30)--(51.5,30);

\draw[edge, thin](-0.5,19)--(5.5,31);

\draw[edge, thin](-.5,-1)--(15.5,31);
\draw[edge, thin](9.5,-1)--(25.5,31);
\draw[edge, thin](19.5,-1)--(35.5,31);
\draw[edge, thin](29.5,-1)--(45.5,31);
\draw[edge, thin](39.5,-1)--(51.5,23);
\draw[edge, thin](49.5,-1)--(51.5,3);


\draw[edge, thin](5.5,-1)--(-0.5,11);
\draw[edge, thin](15.5,-1)--(-0.5,31);
\draw[edge, thin](25.5,-1)--(9.5,31);
\draw[edge, thin](35.5,-1)--(19.5,31);
\draw[edge, thin](45.5,-1)--(29.5,31);

\draw[edge, thin](50.5,9)--(39.5,31);

\draw[edge, thin](51.5,27)--(49.5,31);

\draw [dashed] (22.5,15) -- (32.5,15);
\draw [dashed] (22.5,15) -- (17.5,25);
\draw [dashed] (22.5,15) -- (27.5,25);

\node[ver] () at (22.5,15){\scriptsize $\bullet$};
\node[ver] () at (32.5,15){\scriptsize $\bullet$};
\node[ver] () at (27.5,25){\scriptsize $\bullet$};
\node[ver] () at (17.5,25){\scriptsize $\bullet$};

\node[ver] () at (2.5,-1.7){\scriptsize ${w_{-2,-1}}$};
\node[ver] () at (7.5,1.2){\scriptsize ${v_{-1,-1}}$};
\node[ver] () at (12.5,-1.7){\scriptsize ${w_{-1,-1}}$};
\node[ver] () at (17.5,1.2){\scriptsize ${v_{0,-1}}$};
\node[ver] () at (22.5,-1.7){\scriptsize ${w_{0,-1}}$};
\node[ver] () at (27.5,1.2){\scriptsize ${v_{1,-1}}$};
\node[ver] () at (32.5,-1.7){\scriptsize ${w_{1,-1}}$};
\node[ver] () at (37.5,1.2){\scriptsize ${v_{2,-1}}$};
\node[ver] () at (42.5,-1.7){\scriptsize ${w_{2,-1}}$};
\node[ver] () at (47.5,1.2){\scriptsize ${v_{3,-1}}$};
\node[ver] () at (52.5,-1.7){\scriptsize ${w_{3,-1}}$};

\node[ver] () at (6.3,5){\scriptsize ${u_{-2,-1}}$};
\node[ver] () at (16.3,5){\scriptsize ${u_{-1,-1}}$};
\node[ver] () at (25.5,5){\scriptsize ${u_{0,-1}}$};
\node[ver] () at (35.5,5){\scriptsize ${u_{1,-1}}$};
\node[ver] () at (45.5,5){\scriptsize ${u_{2,-1}}$};

\node[ver] () at (2.9,11){\scriptsize ${v_{-2,0}}$};
\node[ver] () at (7.5,8.7){\scriptsize ${w_{-2,0}}$};
\node[ver] () at (12.5,11){\scriptsize ${v_{-1,0}}$};
\node[ver] () at (17.5,8.7){\scriptsize ${w_{-1,0}}$};
\node[ver] () at (22,11){\scriptsize ${v_{0,0}}$};
\node[ver] () at (27,8.7){\scriptsize ${w_{0,0}}$};
\node[ver] () at (32,11){\scriptsize ${v_{1,0}}$};
\node[ver] () at (37,8.7){\scriptsize ${w_{1,0}}$};
\node[ver] () at (42,11){\scriptsize ${v_{2,0}}$};
\node[ver] () at (47,8.7){\scriptsize ${w_{2,0}}$};
\node[ver] () at (52,11){\scriptsize ${v_{3,0}}$};

\node[ver] () at (5,15){\scriptsize ${u_{-2,0}}$};
\node[ver] () at (15,15){\scriptsize ${u_{-1,0}}$};
\node[ver] () at (30,13.5){\scriptsize ${u_{0,0}}$};
\node[ver] () at (40.5,15){\scriptsize ${u_{1,0}}$};
\node[ver] () at (50,15){\scriptsize ${u_{2,0}}$};

\node[ver] () at (2.5,18.5){\scriptsize ${w_{-1,1}}$};
\node[ver] () at (8,20.7){\scriptsize ${v_{-2,1}}$};
\node[ver] () at (13,18.5){\scriptsize ${w_{-2,1}}$};
\node[ver] () at (18,20.7){\scriptsize ${v_{-1,1}}$};
\node[ver] () at (23,18.5){\scriptsize ${w_{-1,1}}$};
\node[ver] () at (27.5,20.7){\scriptsize ${v_{0,1}}$};
\node[ver] () at (32,18.5){\scriptsize ${w_{0,1}}$};
\node[ver] () at (37.5,20.7){\scriptsize ${v_{1,1}}$};
\node[ver] () at (42,18.5){\scriptsize ${w_{1,1}}$};
\node[ver] () at (47.5,20.7){\scriptsize ${v_{2,1}}$};
\node[ver] () at (52,18.5){\scriptsize ${w_{2,1}}$};

\node[ver] () at (6,25){\scriptsize ${u_{-3,1}}$};
\node[ver] () at (15.5,24){\scriptsize ${u_{-2,1}}$};
\node[ver] () at (25.5,24){\scriptsize ${u_{-1,1}}$};
\node[ver] () at (35.5,25){\scriptsize ${u_{0,1}}$};
\node[ver] () at (45,25){\scriptsize ${u_{1,1}}$};

\node[ver] () at (2.5,31){\scriptsize ${v_{-3,2}}$};
\node[ver] () at (7.7,28.5){\scriptsize ${w_{-3,2}}$};
\node[ver] () at (12.5,31){\scriptsize ${v_{-2,2}}$};
\node[ver] () at (17.5,28.5){\scriptsize ${w_{-2,2}}$};
\node[ver] () at (22.5,31){\scriptsize ${v_{-1,2}}$};
\node[ver] () at (27.5,28.5){\scriptsize ${w_{-1,2}}$};
\node[ver] () at (32.5,31){\scriptsize ${v_{0,2}}$};
\node[ver] () at (37.5,28.5){\scriptsize ${w_{0,2}}$};
\node[ver] () at (42.5,31){\scriptsize ${v_{1,2}}$};
\node[ver] () at (47.5,28.5){\scriptsize ${w_{1,2}}$};
\node[ver] () at (52.5,31){\scriptsize ${v_{2,2}}$};

\node[ver] () at (28,-6) {\small (b): Trihexagonal tiling $E_6$};

\node[ver] () at (2,-8){\normalsize \bf Figure 5};

\end{tikzpicture}


\end{figure}


\vspace{-3mm}


\begin{proof}[Proof of Theorem \ref{theo:torus33434&3636}]
Let $X$ be a semi-equivelar map of type $[3^2,4^1,3^1,4^1]$ on the torus. By similar arguments as in the proof of Theorem \ref{theo:torus} and using Lemma \ref{lemma:plane33434&3636} (a), we assume that there exists a polyhedral covering map $\eta_2 \colon E_2 \to X$.
Let $\Gamma_2$ be the group of covering transformations. Then $|X|
= |E_2|/\Gamma_2$.

Let $V_2$ be the vertex set of $E_2$. We take origin $(0,0)$ is the middle point of the line segment joining $u_{0,0}$ and $u_{1,1}$ (see Fig. 5\,(a)). Let $a = u_{2,0}-u_{0,0}$, $b= u_{0,2}-u_{0,0} \in \mathbb{R}^2$. Consider the translations $x\mapsto x+a$, $x\mapsto x+b$. Then $H_2 := \langle x\mapsto x+a, x\mapsto x+b \rangle$ is the group of all the translations of $E_2$. Under the action of $H_2$, vertices form four orbits.
Consider the subgroup $G_2$ of ${\rm Aut}(E_2)$ generated by $H_2$  and the map (the half rotation) $x\mapsto -x$. So,
\begin{align*}
  G_2 & =\{ \alpha : x\mapsto \varepsilon x + ma + nb \, : \, \varepsilon=\pm 1, m, n\in \ZZ\} \cong H_2\rtimes \mathbb{Z}_2.
\end{align*}

Clearly, under the action of $G_2$, vertices of $E_2$ form two orbits. The two orbits are $O_1=\{u_{i,j} \, : \, i+j \mbox{ is odd}\}$ and
$O_2= \{u_{i,j} \, : \, i+j \mbox{ is even}\}$.

\medskip

\noindent {\em Claim.} If $K \leq H_2$ then $K \unlhd G_2$.

\smallskip

Let $g \in G_2$ and $k\in K$. Assume $g(x) = \varepsilon x+ma+nb$ and $k(x) = x + pa+ qb$ for some $m, n, p, q\in \mathbb{Z}$ and $\varepsilon\in\{1, -1\}$.
Then $(g\circ k\circ g^{-1})(x) = (g\circ k)(\varepsilon(x-ma-nb)) = g(\varepsilon(x-ma-nb)+pa+qb)=x-ma-nb+\varepsilon(pa+qb) +ma+nb= x+\varepsilon(pa+qb) = k^{\varepsilon}(x)$. Thus, $g\circ k\circ g^{-1} = k^{\varepsilon}\in K$. This proves the claim.

\smallskip

For $\sigma\in \Gamma_2$, $\eta_2 \circ \sigma = \eta_2$.
So, $\sigma$ maps a face of the map $E_2$ (in $\mathbb{R}^2$) to a face of $E_2$ (in $\mathbb{R}^2$). This implies that $\sigma$
induces an automorphism $\sigma$ of $E_2$. Thus, we can identify
$\Gamma_2$ with a subgroup of ${\rm Aut}(E_2)$. So, $X$ is a quotient
of $E_2$ by a subgroup $\Gamma_2$ of ${\rm Aut}(E_2)$, where $\Gamma_2$
has no fixed element (vertex, edge or face). Hence $\Gamma_2$
consists of translations and glide reflections. Since $X =
E_2/\Gamma_2$ is orientable, $\Gamma_2$ does not contain any glide
reflection. Thus $\Gamma_2 \leq H_2$. By the claim, $\Gamma_2$ is a normal subgroup of $G_2$. Thus, $G_2/\Gamma_2$ acts on $X= E_2/\Gamma_2$.
Since $O_1$ and $O_2$ are the $G_2$-orbits, it follows that $\eta_2(O_1)$ and $\eta_2(O_2)$ are the $(G_2/\Gamma_2)$-orbits. Since the vertex set of $X$ is $\eta_2(V_2) = \eta_2(O_1) \sqcup \eta_2(O_2)$ and $G_2/\Gamma_2 \leq {\rm Aut}(X)$, part (a) follows.


Let $X$ be a semi-equivelar map of type $[3^1, 6^1, 3^1, 6^1]$ on the torus. By similar arguments as in the proof of Theorem \ref{theo:torus} and using Lemma \ref{lemma:plane33434&3636} (b), we assume that there exists a polyhedral covering map $\eta_6 \colon E_6 \to X$.
Let $\Gamma_6$ be the group of covering transformations. Then $|X|= |E_6|/\Gamma_6$.

Let $V_6$ be the vertex set of $E_6$. We take origin $(0,0)$ is the middle point of the line segment joining $u_{-1,0}$ and $u_{0,0}$ (see Fig. 5\,(b)). Let
$r = u_{1,0}-u_{0,0}=v_{1,0}-v_{0,0}= w_{1,0}- w_{0,0}$,
$s = u_{0,1}-u_{0,0}= v_{0,1}-v_{0,0}= w_{0,1}- w_{0,0}$ and
$t = u_{-1,1}-u_{0,0}= v_{-1,1}-v_{0,0}= w_{-1,1}- w_{0,0}$. Consider the translations $x\mapsto x+r$, $x\mapsto x+s$ and $x\mapsto x+t$. Then $H_6 := \langle x\mapsto x+r, x\mapsto x+s, x\mapsto x+t \rangle$ is the group of all the translations of $E_6$.
Since $H_6$ is a group of translations it is abelian.
Under the action of $H_6$, vertices form three orbits. The orbits are $O_u =\{u_{i,j} \, : \, i, j \in \mathbb{Z}\}$, $O_v=\{v_{i,j} \, : \, i, j \in \mathbb{Z}\}$, $O_w=\{w_{i,j} \, : \, i, j \in \mathbb{Z}\}$.

As before, we can identify $\Gamma_6$ with a subgroup of $H_6$. So, $X$ is a quotient of $E_6$ by a group $\Gamma_6$, where $\Gamma_6 \leq H_6\leq {\rm Aut}(E_6)$.
Since $H_6$ is abelian, $\Gamma_6$ is a normal subgroup of $H_6$. Thus, $H_6/\Gamma_6$ acts on $X= E_6/\Gamma_6$.
Since $O_u$, $O_v$ and $O_w$ are the $H_6$-orbits, it follows that $\eta_6(O_u)$, $\eta_6(O_v)$ and $\eta_6(O_w)$ are the $(H_6/\Gamma_6)$-orbits. Since the vertex set of $X$ is $\eta_6(V_6) = \eta_6(O_u) \sqcup \eta_6(O_v) \sqcup \eta_6(O_w)$ and $H_6/\Gamma_6 \leq {\rm Aut}(X)$, part (b) follows.
\end{proof}


\medskip

\noindent {\bf Acknowledgements:}
The first author is supported by DST, India (DST/INT/AUS/P-56/2013(G)) and DIICCSRTE, Australia (project AISRF06660)  under the Australia-India Strategic Research Fund. The second author is supported by NBHM, India for Post-doctoral Fellowship (2/40(34)/2015/R\&D-II/11179). The authors thank the
anonymous referee for some useful comments and for drawing their attention to the reference \cite{BK2008}.


{\small

}


\begin{thebibliography}{99}
\bibitem{Ba1991}
L. Babai, {\em Vertex-transitive graphs and vertex-transitive maps}, J. Graph Theory {\bf 15} (1991), 587--627.

\bibitem{BK2008}
U. Brehm and W. K\"{u}hnel, {\em Equivelar maps on the torus}, European J. Combin. {\bf 29} (2008), 1843--1861.

\bibitem{CM1957} H. S. M. Coxeter and W. O. J. Moser, {\em Generators and Relations for Discrete Groups} (4th edition), Springer-Verlag, Berlin-New York, 1980.

\bibitem{Da2005}
B. Datta, {\em A note on the existence of $\{k, k\}$-equivelar polyhedral maps}, {Beitr\"{a}ge Algebra Geom.} {\bf 46} (2005), 537--544.

\bibitem{DN2001}
B. Datta and N. Nilakantan, {\em Equivelar polyhedra with few vertices}, {Discrete \& Comput Geom.} {\bf 26} (2001), 429--461.

\bibitem{DU2005}
B. Datta and A. K. Upadhyay, {\em Degree-regular triangulations of torus and Klein bottle}, Proc. Indian Acad. Sci. (Math. Sci.) {\bf 115} (2005), 279--307.

\bibitem{FT1965} L. Fejes T\'{o}th, {\em Regul\"{a}re Figuren}, Akad\'{e}miai Kiad\'{o}, Budapest, 1965. (English translation: {\em Regular Figures}, Pergmon Press, Oxford, 1964.
%
\bibitem{GS1977}
B. Gr\"{u}nbaum and G. C. Shephard, {\em Tilings by regular polygons: Patterns in the plane from Kepler to the present, including recent results and unsolved problems}, Math. Mag. {\bf 50} (1977), 227--247.

%
\bibitem{GS1981}
B. Gr\"{u}nbaum and G. C. Shephard, {\em The geometry of planar graphs.}
Combinatorics (Swansea, 1981), pp. 124-–150, London Math. Soc. LNS {\bf 52},
Cambridge Univ. Press, Cambridge, 1981.

%
\bibitem{MU2015} D. Maity and A. K. Upadhyay, {\em On enumeration of a class of maps on Klein bottle}, arXiv:1509.04519v2 [math.CO].
%
\bibitem{Sp1966}
E. H. Spanier, {\em Algebraic Topology}, Springer-Verlag, New York, 1966.
%
\bibitem{Su2011t}
O. \v{S}uch, {\em Vertex transitive maps on a torus}, Acta Math. Univ. Comenianae {\bf 53} (2011), 1--30.
%
\bibitem{Su2011kb}
O. \v{S}uch, {\em Vertex transitive maps on the Klein bottle}, {ARS Math. Contemporanea} {\bf 4} (2011), 363--374.

\end{thebibliography}
\end{document}